\documentclass[10pt]{amsart}
\usepackage{latexsym, amsmath,amssymb}

\usepackage{hyperref}

\renewcommand{\epsilon}{\varepsilon}
\newcommand{\newsection}[1]
{\subsection{#1}\setcounter{theorem}{0} \setcounter{equation}{0}
\par\noindent}

\newtheorem{theorem}{Theorem}

\newtheorem{lemma}[theorem]{Lemma}
\newtheorem{corr}[theorem]{Corollary}

\newtheorem{proposition}[theorem]{Proposition}
\newtheorem{deff}[theorem]{Definition}

\newcommand{\bth}{\begin{theorem}}
\newcommand{\ble}{\begin{lemma}}
\newcommand{\bcor}{\begin{corr}}

\newcommand{\bdeff}{\begin{deff}}

\newcommand{\bprop}{\begin{proposition}}
\newcommand{\ele}{\end{lemma}}
\newcommand{\ecor}{\end{corr}}
\newcommand{\edeff}{\end{deff}}

\newcommand{\eprop}{\end{proposition}}

\newcommand{\Rn}{{\mathbb R}^n}

\newcommand{\la}{\lambda}

\newcommand{\e}{\varepsilon}

\renewcommand{\Pi}{\varPi}

\renewcommand{\epsilon}{\varepsilon}

\newcommand{\R}{{\mathbb R}}

\newcommand{\1}{{\rm 1\hspace*{-0.4ex}%
\rule{0.1ex}{1.52ex}\hspace*{0.2ex}}}

\newcommand{\ola}{\1_\la}

\newcommand{\ala}{\tilde \1_\la}

\begin{document}

\title[Weyl formulae for Schr\"odinger operators]
{Weyl formulae for Schr\"odinger operators with critically singular potentials}
%
%
%
%
%
%

\thanks{The authors were supported in part by the NSF (NSF Grant DMS-1665373), and the second author was also partially supported by
   the Simons Foundation. }

\author[]{Xiaoqi Huang}
\address[X.H.]{Department of Mathematics,  Johns Hopkins University,
Baltimore, MD 21218}
\email{xhuang49@math.jhu.edu}

\author[]{Christopher D. Sogge}
\address[C.D.S.]{Department of Mathematics,  Johns Hopkins University,
Baltimore, MD 21218}
\email{sogge@jhu.edu}

\begin{abstract}We obtain generalizations of classical versions of the Weyl formula involving Schr\"odinger
operators $H_V=-\Delta_g+V(x)$ on compact boundaryless Riemannian manifolds with critically singular
potentials $V$.  In particular, we extend the classical results of Avakumovi\'{c}~\cite{Avakumovic}, Levitan~\cite{Levitan} and H\"ormander~\cite{HSpec} by obtaining $O(\la^{n-1})$ bounds for the error term
in the Weyl formula in the universal case when we assume that $V\in L^1(M)$ with the negative part $V^{-}=\max\{0, -V\}$ belongs to the Kato class,
${\mathcal K}(M)$, which is the minimal assumption to ensure that $H_V$ is essentially self-adjoint and
bounded from below or has favorable heat kernel bounds.  In this case, we can also obtain extensions of the
Duistermaat-Guillemin~\cite{DuistermaatGuillemin} theorem yielding $o(\la^{n-1})$ bounds for the error term under generic
conditions on the geodesic flow, and we can also extend B\'erard's~\cite{Berard} theorem yielding
$O(\la^{n-1}/\log \la)$ error bounds under the assumption that the principal curvatures are non-positive
everywhere.  We can obtain further improvements for tori, which are essentially optimal, if we strengthen
the assumption on the potential to $V\in L^p(M)$ and $V^{-}\in {\mathcal K}(M)$ for appropriate exponents
$p=p_n$.
\end{abstract}

\maketitle
\setcounter{secnumdepth}{3}

\newsection{Introduction}


The purpose of this paper is to prove Weyl formulae for Schr\"odinger operators
\begin{equation}\label{1.1}
H_V=-\Delta_g+V(x)
\end{equation}
on smooth compact $n$-dimensional Riemannian manifolds $(M,g)$.  We shall assume throughout that the potentials
$V$ are real-valued.  Moreover, we shall assume that 
\begin{equation}\label{1.2}
V\in L^1(M),  \,\,\, \text{and} \,\,\, V^{-}=\max\{0, -V\}\in  {\mathcal K}(M),
\end{equation}
where ${\mathcal K}(M)$ denotes the Kato class.  Recall that ${\mathcal K}(M)$ is all $V$ satisfying
\begin{equation}\label{1.3}
\lim_{\delta\to 0}\Bigl(\, \sup_{x\in M} \int_{B(x,\delta)}|V(y)| \, h_n(d_g(x,y)) \, dy\, \Bigr)=0,
\end{equation}
where $d_g$, $dy$ and $B(x,\delta)$ denote geodesic distance, the volume element and the geodesic ball of radius $\delta$ about
$x$ associated with the metric $g$ on $M$, respectively, and
\begin{equation}\label{1.4}
h_n(r)=
\begin{cases}
r^{2-n}, \quad n\ge 3
\\
\log(2+1/r), \quad n=2.
\end{cases}
\end{equation}
Note that our condition on $V$ in \eqref{1.2} is weaker than $V\in {\mathcal K}(M)$ since we only requires the positive part $V^{+}\in L^1$ while 
${\mathcal K}(M)\subset L^1(M)$.

As was shown in \cite{BSS} (see also \cite{SimonSurvey})
the assumption that $V^{-}$ is in the Kato class is needed
to ensure that
the eigenfunctions of $H_V$ are bounded. If $H_V$ has
unbounded eigenfunctions, then its spectral projection
kernels will be unbounded for large enough $\lambda$, and
obtaining spectral bounds in this situation seems far-fetched.  The assumption that $V^{-}\in {\mathcal K}(M)$ ensures that  this is not the case.  

Moreover, if $V$ is as in \eqref{1.2} then the Schr\"odinger operator $H_V$ in \eqref{1.1} is self-adjoint and bounded from below (see e.g. \cite{lorinczi2011feynman},\cite{guneysu2017heat}). Additionally, in this case, since $M$ is compact, the spectrum of $H_V$ is discrete.  Also, (see \cite{SimonSurvey}) the associated
eigenfunctions are continuous.  Assuming, as we may, that
$H_V$ is a positive operator, we shall write the spectrum
of $\sqrt{H_V}$  as
\begin{equation}\label{1.5}
\{\tau_k\}_{k=1}^\infty,
\end{equation}
where the eigenvalues, $\tau_1\le \tau_2\le \cdots$, are arranged in increasing order and we account for multiplicity.  For each $\tau_k$ there is an 
eigenfunction $e_{\tau_k}\in \text{Dom }(H_V)$ (the domain of $H_V$) so that
\begin{equation}\label{1.6}
H_Ve_{\tau_k}=\tau^2_ke_{\tau_k}.
\end{equation} 
We shall always assume that the eigenfunctions are
$L^2$-normalized, i.e.,
$$\int_M |e_{\tau_k}(x)|^2 \, dx=1.$$

After possibly adding a constant to $V$ we may, and shall, assume throughout that $H_V$ is bounded below by one, i.e.,\footnote{See the remark after Theorem~\ref{DG}.}
\begin{equation}\label{1.7}
\|f\|_2^2\le \langle \, H_Vf, \, f\, \rangle, \quad
f\in \text{Dom }(H_V).
\end{equation}
Also, to be consistent, we shall let
\begin{equation}\label{1.8}
H^0=-\Delta_g+1
\end{equation}
be the unperturbed operator also enjoying this lower
bound.  The corresponding eigenvalues and associated $L^2$-normalized
eigenfunctions are denoted by $\{\lambda_j\}_{j=1}^\infty$ and $\{e^0_j\}_{j=1}^\infty$, respectively so
that
\begin{equation}\label{1.9}
H^0e^0_j=\lambda^2_j e^0_j, \quad \text{and }\, \, 
\int_M |e^0_j(x)|^2 \, dx=1.
\end{equation}

Both $\{e_{\tau_k}\}_{k=1}^\infty$ and $\{e^0_j\}_{j=1}^\infty$ are orthonormal bases for $L^2(M)$.  Recall (see e.g. \cite{SoggeHangzhou}) that if $N^0(\la)$ denotes the 
Weyl counting function for $H^0$ then one has the 
``sharp Weyl formula''
\begin{equation}\label{1.10}
N^0(\la)=(2\pi)^{-n} \omega_n \text{Vol}_g(M) \, \la^n
\, +\, O(\la^{n-1}), 
\quad N^0(\la)=\# \{j: \, \la_j\le \la\},
\end{equation}
where $\omega_n$ denotes the volume of the unit ball in
$\Rn$ and $\text{Vol}_g(M)$ denotes the Riemannian volume of $M$.  This result is due to Avakumovi\'{c}~\cite{Avakumovic} and Levitan~\cite{Levitan}, and it was generalized to general self-adjoint elliptic pseudo-differential operators
by H\"ormander~\cite{HSpec}.  The bound in \eqref{1.10}
cannot be improved for the standard round sphere, which accounts for the nomenclature ``sharp Weyl formula''.

The main goal of this paper is to show that this sharp Weyl formula also holds for the operators $H_V$ in
\eqref{1.1} involving critically singular potentials
$V$ as in \eqref{1.2}.  Specifically we have the following.

\begin{theorem}\label{thm1.1}  Let  $V\in L^1(M)$, and $V^{-}\in  {\mathcal K}(M)$, let $H_V$
as above and set
\begin{equation}\label{1.11}
N_V(\la)=\#\{k: \, \tau_k\le \la\}.
\end{equation}
We then have
\begin{equation}\label{1.12}
N_V(\la)=(2\pi)^{-n}\omega_n \mathrm{Vol}_g(M) \, \la^n
+O(\la^{n-1}).
\end{equation}
\end{theorem}

Note that in the case $n=3$,  shortly after an earlier version of this paper, a local pointwise version of \eqref{1.12} with the same sharp reminder term was established
in \cite{frank2020sharp} by Frank and Sabin, under a slightly stronger condition than Kato class for the potentials. They also showed that the pointwise Weyl law can be 
violated for Kato potentials.

We shall also be able to obtain improved counting
estimates under certain geometric assumptions. 

The first such result is an extension of the Duistermaat-Guillemin theorem \cite{DuistermaatGuillemin}.  Recall the assumption in this theorem is that the set 
${\mathcal C}\subset S^*M$ of all $(x,\xi)$ 
lying on a periodic geodesic in $S^*M$ have 
measure zero (see \cite{SoggeHangzhou}).  Here,
$S^*M$ denotes the unit cotangent bundle of $(M,g)$.  In this case Duistermaat and Guillemin~\cite{DuistermaatGuillemin} showed that one can improve the bounds for the error term in the Weyl law \eqref{1.10} (assuming that 
$V=0$ or $V$ is smooth) to be $o(\la^{n-1})$.
The proof of this relies on H\"ormander's theory of propagation of singularities for smooth pseudo-differential operators.  Even though this theory does not apply to our situation involving very singular potentials, we can extend the theorem of Duistermaat and Guillemin to include the above operators.

\begin{theorem}\label{DG}
Let  $V\in L^1(M)$, and $V^{-}\in  {\mathcal K}(M)$, let $H_V$ be as above, and assume that the set ${\mathcal C}$ of directions of periodic geodesics has measure zero in $S^*M$.  
Then
\begin{equation}\label{1.13}
N_V(\la)=(2\pi)^{-n}\omega_n \mathrm{Vol}_g(M) \, \la^n
+o(\la^{n-1}).
\end{equation}
\end{theorem}

\medskip

\noindent {\bf Remark.}  Note that, under the hypothesis of Theorem~\ref{DG}, the special case where $V\equiv0$
implies, for instance, the results where $V\equiv c\ge 0$.  
For if $N(\la)$ denotes the number of eigenvalues of $\sqrt{-\Delta_g}$ which are $\le \la$ then, under the hypotheses of Theorem~\ref{DG},  by
 \cite[Theorem 3.5]{DuistermaatGuillemin}, we have $N(\la)=(2\pi)^{-n}\omega_n \mathrm{Vol}_g(M) \la^n+o(\la^{n-1})$.
Based on this we see that the results for $V\equiv0$ imply the results where $V\equiv c\ge 0$ since
$N_{V\equiv c}(\lambda)=N(\sqrt{\lambda^2-c})$ and $(\sqrt{\lambda^2-c})^n=\la^n \cdot(\sqrt{1-c\la^{-2}})^n
=\la^n+O(\la^{n-2})=\la^n+o(\la^{n-1})$.  Moreover, if $V\in C^\infty(M)$ then \eqref{1.13} follows directly from  Duistermaat-Guillemin  \cite[Theorem 3.5]{DuistermaatGuillemin} since $\sqrt{-\Delta_g}$ and $\sqrt{H_V}$ have the same principal symbol
and zero subprincipal symbol.   See \cite[p. 41]{DuistermaatGuillemin} for the latter.  As we shall see in what follows, error terms of order $O(\la^{n-2})$ occur recurrently in our arguments.

\medskip

We also can extend the classical theorem of
B\'erard~\cite{Berard}.

\begin{theorem}\label{berardthm}
Assume that the sectional curvatures of $(M,g)$ are
non-positive.  Then, if $V\in L^1(M)$, and $V^{-}\in  {\mathcal K}(M)$,
\begin{equation}\label{1.14}
N_V(\la)=(2\pi)^{-n}\omega_n \mathrm{Vol}_g(M) \, \la^n
+O(\la^{n-1}/\log \la).
\end{equation}
\end{theorem}

In the special case of the torus, we can do much better.

\begin{theorem}\label{torusthm}
Let ${\mathbb T}^n=\Rn/{\mathbb Z}^n$ denote the standard torus with the flat metric, and assume that
$V\in {\mathcal K}(M)$ when $n=2$ and $V\in L^p(M)$, $V^{-}\in{\mathcal K}(M)$ for some $p>\tfrac{2n}{n+2}$
if $n\ge 3$.
Then 
\begin{equation}\label{1.15}
N_V(\la)=(2\pi)^{-n}\omega_n \mathrm{Vol}_g(M) \, \la^n
+O(\la^{n-2+2/(n+1)}).
\end{equation}
Moreover, if  $V\in L^{2}(M)$ and $V^{-}\in \mathcal{K}(M)$, we have for $n\geq 4$, 
\begin{equation}\label{1.16}
N_V(\la)=(2\pi)^{-n}\omega_n \mathrm{Vol}_g(M) \, \la^n
+O(\la^{n-2+\e}), \quad \forall \, \e>0.
\end{equation}
\end{theorem}

If $V\equiv 1$, the bounds in \eqref{1.15} are the classical results of Hlwaka~\cite{Hlawka}. The same bounds hold for irrational tori.  Also, with the stronger condition 
on the potential $V$ in the second part of the Theorem, if we use more recent improved bounds for the error term in the Weyl formula for $V\equiv 1$ and related bounds for the trace of certain spectral projection operators,  we obtain the improved bounds in  \eqref{1.16} involving singular potentials.

We mention that the proof of Theorem~\ref{thm1.1}-\ref{torusthm} is based on a perturbative argument which focus on estimating the difference between the counting functions for perturbed and 
unperturbed case, where the differences only contribute to the error terms in the main theorems. As a result, the arguments in this paper should also allow us to obtain improved counting estimates for 
the Schr\"odinger operators $H_V$ under more general geometric assumptions, see, e.g., \cite{volovoy1990improved}, \cite{volovoy1990verification}, as well as more recent works \cite{iosevich2019weyl}, \cite{canzani2020weyl}, for more details.

The authors are grateful to the referee for several suggestions that improved the exposition.

\newsection{Preliminaries and an abstract universal bound}\label{preliminaries}

The purpose of this section is to introduce a abstract proposition that will allow us to prove Theorem 1.1-1.3, provided that we have the analogous counting estimates for the unperturbed operators $H_0=-\Delta_g$. The proof of Theorem~\ref{torusthm}, which is based on a slightly different method that requires stronger conditions on $V$, is given separately in the last section.

Throughout this section we shall assume that, if $N^0(\la)$ denotes the Weyl counting function for $H^0$, we have
\begin{equation}\label{a.1}
N^0(\la)=\int_M\sum_{\la_j\le \la}
|e^0_j(x)|^2 \, dx=(2\pi)^{-n}\omega_n \mathrm{Vol}_g(M) \, \la^n
+O(\e\la^{n-1}),
\end{equation}
where  $\e=\e(\la)$ is a non-increasing function in $\la$ which satisfies $\e(2\la)\ge \frac12\e(\la)$ and $0<\e(\la)\le 1$, $\forall \,\la\ge1$.
The assumption on $\e(\la)$ is a very mild one, which, for instance is satisfied whenever $\e(\la)=\la^{-\sigma}$ with
$0\le \sigma\le 1$.  We make this assumption, since, as we  mentioned before, error terms of the form $O(\la^{n-2})$ arise repeatedly.

The abstract proposition we shall need is the following
\begin{proposition}\label{abstract}
 Let  $V\in L^1(M)$, and $V^{-}\in  {\mathcal K}(M)$, let $H_V$
as above, if $N^0(\la)$ satisfies \eqref{a.1}, then for the same $\e=\e(\la)$ appearing  in \eqref{a.1}, we have 
\begin{equation}\label{a.2}
N_V(\la)=(2\pi)^{-n}\omega_n \mathrm{Vol}_g(M) \, \la^n
+O(\e\la^{n-1}+\e^{-1}\la^{n-\frac32}).
\end{equation}
\end{proposition}

Note that as a consequence of \eqref{1.10}, \eqref{a.1} holds with $\e\equiv 1$, thus Theorem~\ref{thm1.1} follows directly from \eqref{a.2}. Similarly, under the hypothesis of Theorem~\ref{DG}, 
for any fixed constant $T\gg 1$, \eqref{a.1} holds with $\e=1/T$ if $\la\ge\Lambda(T)$. Since $T\la^{n-\frac32}$ is bounded by $1/T\la^{n-1}$ for sufficiently large $\la$, \eqref{a.2} also implies 
\eqref{1.13}. Additionally, under the hypothesis of Theorem~\ref{berardthm}, by the classical theorem of B\'erard~\cite{Berard}, \eqref{a.1} holds with $\e=1/\log\la$, and since $\la^{n-\frac32}\log\la$ is bounded by $\la^{n-1}(\log\la)^{-1}$ for sufficiently large $\la$, Theorem~\ref{berardthm} also follows from the above proposition. On the other hand, if $\e\le\la^{-\frac14}$, the reminder term $\e^{-1}\la^{n-\frac32}$ would be too large compared with $\e\la^{n-1}$, thus \eqref{a.2} is not sufficient to give us the results in Theorem~\ref{torusthm}.

To prove \eqref{a.2}, we first recall that, if as above, $\{e_{\tau_k}\}$ is an orthonormal basis of eigefunctions of $H_V$ then
\begin{equation}\label{2.1}
N_V(\la)=\# \{k: \, \tau_k\le \la\} =
\int_M \sum_{\tau_k\le \la} \, |e_{\tau_k}(x)|^2 \, dx.
\end{equation}
Thus, $N_V(\la)$ is the {\em trace} of the spectral function
\begin{equation}\label{2.2}
E^V_\la(x,y)=\sum_{\tau_k\le \la}e_{\tau_k}(x) e_{\tau_k}(y).
\end{equation}
Here, we are assuming, as we may, that all the eigenfunctions of $H_V$ in our orthonormal basis are 
real-valued.  To simplify the notation, as we may, we
shall assume the same for those of $H^0$, i.e.,
the $\{e_j^0\}$.

We shall need the following lemma which is key to the proof of \eqref{a.2}.

\begin{lemma}\label{smallinterval}
Let $\chi_{\la,\e}^V(x,y)$ be the kernels of the
spectral projection operators
\begin{equation}\label{2.3}
\chi_{\la,\e}^V(x,y)=\sum_{\tau_k\in [\la,\la+\e)}e_{\tau_k}(x)e_{\tau_k}(y)
\end{equation}
for $H_V$, with $\e=\e(\la)$ defined as in \eqref{a.1}. Then given \eqref{a.1}, we have
\begin{equation}\label{2.4}
\int_M \chi_{\la,\e}^V(x,x)\, dx =O(\e\la^{n-1}+\e^{-1}\la^{n-\frac32}), \, \, 
\la \ge 1.
\end{equation}
\end{lemma}

The left side of \eqref{2.4} is essentially equal to the number of eigenvalues for the operator $\sqrt{H_V}$ inside the interval $[\la,\la+\e]$. Note that when $V\equiv 1$, as a consequence of \eqref{a.1}, if
$$\chi_{\la,\e}^0(x,y)=\sum_{\la_j\in [\la,\la+\e)}e_j^0(x)e_j^0(y),$$
denotes the spectral projection operator onto the interval $ [\la,\la+\e)$, we have
\begin{equation}\label{a.3}
\int_M \chi_{\la,\e}^0(x,x)\, dx =O(\e\la^{n-1}), \, \, 
\la \ge 1.
\end{equation}
We shall postpone the proof of Lemma~\ref{smallinterval} to the end of next section, and first see how we can apply it to the proof of \eqref{a.2}.

To make use of the above lemma, we shall follow the classical approach
of rewriting the traces using the wave equation.  To this end, let $P^0=\sqrt{H^0}$ and $P_V=\sqrt{H_V}$ be
the square roots of the two Hamiltonians.  Then since
the Fourier transform of the indicator function
$\1_\la(\tau)$ is $2\tfrac{\sin \la t}t$, we have for $\lambda$ not in the spectrum of $P^0$
\begin{equation}\label{2.5}
N^0(\la)=\frac1\pi \int_M \int_{-\infty}^\infty
\frac{\sin t\la}t \, \bigl(\cos tP^0\bigr)(x,x) \, dt dx,
\end{equation}
if
\begin{equation}\label{2.6}
\bigl(\cos (tP^0)\bigr)(x,y)=\sum_j \cos t\la_j e_j^0(x)
e_j^0(y)
\end{equation}
is the kernel of the solution operator for
$f\to (\cos tP^0)f=u^0(t,x)$, where $u^0$ solves the wave equation
\begin{equation}\label{2.7}
(\partial_t^2+H^0)u^0(x,t)=0, \, \, (x,t)\in M\times \R,
\, \, u^0|_{t=0}=f, \, \, \partial_tu^0|_{t=0}=0.
\end{equation}
Note that \eqref{2.6} is the kernel of a bounded operator on $L^2(M)$, and when we check that \eqref{2.7} is valid, it suffices to do so when $f$ is a finite linear combination of the $\{e_j^0\}$ since such functions are dense in $L^2(M)$.  We shall use similar facts in what follows.   See \cite{SoggeHangzhou} for more details.

Similarly,  for $\lambda$ not in the spectrum of $P_V$
\begin{equation}\label{2.8}
N_V(\la)=\frac1\pi \int_M\int_{-\infty}^\infty
\frac{\sin t\la}t \, \bigl(\cos(tP_V)\bigr)(x,x) \, dt dx,
\end{equation}
if 
\begin{equation}\label{2.9}
\bigl(\cos(tP_V)\bigr)(x,y)=\sum_k \cos t\tau_k \, e_{\tau_k}(x)e_{\tau_k}(y)
\end{equation}
is the kernel of $f\to \cos (tP_V)f=u_V(x,t)$, where 
$u_V$ solve the wave equation
\begin{equation}\label{2.10}
(\partial_t^2+H_V)u_V(x,t)=0, \, \, (x,t)\in M\times \R,
\, \, u_V|_{t=0}=f, \, \, \partial_tu_V|_{t=0}=0.
\end{equation}

To exploit \eqref{a.1} and prove its more general
version \eqref{1.11}, in view of \eqref{2.5}--\eqref{2.10}, it will be useful to relate the kernels in \eqref{2.6} and \eqref{2.9}.  To do so we shall make
use of the following simple calculus lemma.

\begin{lemma}\label{triglemma}  If $\mu\ne \tau$ we have
\begin{equation}\label{2.11}\int_0^t \frac{\sin(t-s)\mu}\mu\,  \cos s\tau \, ds
=\frac{\cos t\tau- \cos t\mu}{\mu^2-\tau^2}.\end{equation}
Similarly, 
\begin{equation}\label{2.12}
\int_0^t \frac{\sin(t-s)\tau}\tau  \, \cos s\tau
\, ds =
\frac{t\sin t\tau}{2\tau}.
\end{equation}
\end{lemma}

\begin{proof}
To prove \eqref{2.11} we make use of the identity
$$\sin\bigl(s(\tau-\mu)+t\mu\bigr)=\sin\bigl((t-s)\mu+s\tau)=
\sin((t-s)\mu)\cos s\tau + \cos((t-s)\mu)\sin s\tau,
$$
and, similarly,
$$-\sin\bigl((\tau+\mu)s-t\mu\bigr)=\sin\bigl((t-s)\mu-s\tau\bigr)
=\sin((t-s)\mu)\cos s\tau -\cos((t-s)\mu) \sin s\tau.$$
Thus,
$$\sin\bigl(s(\tau-\mu)+t\mu\bigr)-\sin\bigl((\tau+\mu)s-t\mu\bigr)=2\sin((t-s)\mu)\cos s\tau.$$
Consequently, the left side of \eqref{2.11} equals
\begin{align*}
\frac1{2\mu}\cdot\Biggl[ \, \frac{\cos\bigl(s(\tau-\mu)+t\mu\bigr)}{\mu-\tau}
&+\frac{\cos\bigl(s(\tau+\mu)-t\mu\bigr)}{\mu+\tau}\, \Biggr]^t_0
\\
&=\frac1{2\mu}\Bigl[ \, \cos t\tau \cdot \Bigl(\frac1{\mu-\tau}+\frac1{\mu+\tau}\Bigr)
-\cos t\mu \cdot  \Bigl(\frac1{\mu-\tau}+\frac1{\mu+\tau}\Bigr)\Bigr]
\\
&=\frac1{2\mu}\cdot 
\Bigl(\frac{2\mu \cos t\tau}{\mu^2-\tau^2} - \frac{2\mu \cos t\mu}{\mu^2-\tau^2}\Bigr)
=\frac{\cos t\tau- \cos t\mu}{\mu^2-\tau^2},
\end{align*}
as desired.  

The proof of \eqref{2.12} is similar.
\end{proof}

Let us now describe how we shall use \eqref{a.1} and 
Lemma~\ref{triglemma} to prove the Weyl formula \eqref{a.2}.  If, as above, $\ola(\tau)$ is the indicator function of $[-\la,\la]$, by \eqref{2.1},
proving this amounts to showing that the trace of 
$\ola(P_V)$ satisfies the bounds in \eqref{1.12}.
As is the custom (cf. \cite{SoggeHangzhou}), we shall
do this indirectly by showing that an $\e=\e(\la)$-dependent approximation
$\ala(P_V)$ also enjoys these bounds, and, separately
showing that the difference between the trace of
$\ola(P_V)$ and $\ala(P_V)$ is $O(\e\la^{n-1}+\e^{-1}\la^{n-\frac32})$.

To this end, fix an even real-valued function $\rho\in 
C^\infty(\R)$ satisfying
\begin{equation}\label{2.13}
\rho(t)=1 \, \, \text{on } \, \, 
[-1/2,1/2] \, \, \, \text{and } \, \,
\text{supp } \rho \subset (-1,1).
\end{equation}
We then define
\begin{equation}\label{2.15}
\ala(\tau)=\frac1\pi \int_{-\infty}^\infty
\rho(\e t)\frac{\sin \la t}t \, \cos t\tau \, dt.
\end{equation}
Then since the Fourier tranform of $\ola(\tau)$ is
$2\tfrac{\sin \la t}t$ it is not difficult to see that for $\tau>0$ and large $\lambda$ we have
\begin{equation}\label{2.16}
\ola(\tau)-\ala(\tau) = O\bigl((1+\e^{-1}|\la-\tau|)^{-N}\bigr)
\, \, \, \forall \, N.
\end{equation}
Also, for later use, for $\tau>0$ we have
\begin{equation}\label{2.17}
\bigl(\tfrac{d}{d\tau}\bigr)^j \,  \ala(\tau)
= O\bigl(\e^{-j}(1+\e^{-1}|\la-\tau|)^{-N}\bigr)
\, \, \, \forall \, N, \quad
\text{if } \, j=1,2,3,\dots.
\end{equation}

If we use Lemma~\ref{smallinterval}, we can estimate the difference
between the trace of $\ola(P_V)-\ala(P_V)$.  Indeed,
by \eqref{2.16} we have
\begin{multline}\label{2.18}
\Bigl| \, \int_M \bigl(\ola(P_V)(x,x)-\ala(P_V)(x,x)\bigr) \, dx \, \Bigr|
=\Bigl|\, \int_M \sum_k\bigl(\ola(\tau_k)-\ala(\tau_k)\bigr) \, |e_{\la_j}(x)|^2 \, dx \, \Bigr|
\\
\lesssim \sum_k \int_M\bigl(1+\e^{-1}|\la-\tau_k|)^{-2n} \, |e_{\tau_k}(x)|^2 \, dx
\lesssim \e\la^{n-1}+\e^{-1}\la^{n-\frac32},
\end{multline}
using \eqref{2.4} as well as the condition on $\e=\e(\la)$ in \eqref{a.1} in the last inequality.  Here, and in what follows, we are using
the notation that $A\lesssim B$ means that $A$ is less
than or equal to a constant times $B$ where the constant
may change at each occurrence.  

Similarly, by \eqref{a.3}, we have 
\begin{multline}\label{a.4}
\Bigl| \, \int_M \bigl(\ola(P^0)(x,x)-\ala(P^0)(x,x)\bigr) \, dx \, \Bigr|
=\Bigl|\, \int_M \sum_j\bigl(\ola(\la_j)-\ala(\la_j)\bigr) \, |e_{\tau_k}(x)|^2 \, dx \, \Bigr|
\\
\lesssim \sum_j \int_M\bigl(1+\e^{-1}|\la-\la_j|)^{-2n} \, |e_{\la_j}(x)|^2 \, dx
\lesssim \e\la^{n-1}.
\end{multline}

Thus, in view of 
\eqref{2.18}, \eqref{a.4}, and \eqref{a.1},  in order to prove
Proposition~\ref{abstract}, it suffices to prove our main
estimate
\begin{equation}\label{2.19}
\int_M \, \Bigl(\ala(P_V)(x,x)-\ala(P^0)(x,x)\Bigr)
\, dx =O(\e^{-1}\la^{n-\frac32}).
\end{equation}
The implicit constants here of course depend
on our $V$ as in Proposition~\ref{abstract}.

To prove this, we shall use the fact that, by \eqref{2.9} and \eqref{2.15} the
kernel of $\ala(P_V)$ is
\begin{equation}\label{2.20}\ala(P_V)(x,y)=
\frac1\pi \int_{-\infty}^\infty
\rho(\e t)\frac{\sin \la t}t \, \sum_k\cos t\tau_k \, e_{\tau_k}(x)
e_{\tau_k}(y) \, dt.\end{equation}

To use this formula, we note that, by \eqref{2.10} if $f$
is a finite combination of the $\{e_{\tau_k}\}$, then
\begin{multline*}
\bigl(\partial_t^2 +H^0\bigr) \int_M
\sum_k \cos t\tau_k e_{\tau_k}(x)e_{\tau_k}(y) \, f(y)\, dy
\\
=-V(x) \cdot \int_M
\sum_k \cos t\tau_k e_{\tau_k}(x)e_{\tau_k}(y) \, f(y)\, dy
=-V(x)\cdot \bigl(\cos tP_V\bigr)(f)(x).
\end{multline*}
Also, since
\begin{multline*}\Bigl(\frac{d}{dt}\Bigr)^j
\Bigl(\, \int_M
\sum_k \cos t\tau_k e_{\tau_k}(x)e_{\tau_k}(y) \, f(y)\, dy
- \int_M
\sum_j \cos t\la_j e_j^0(x)e_j^0(y) \, f(y)\, dy\, \Bigr)\Big|_{t=0}
=0, \\ j=0,1,
\end{multline*}
by Duhamel's principle we have
\begin{align*}
\int_M
&\sum_k \cos t\tau_k e_{\tau_k}(x)e_{\tau_k}(y) \, f(y)\, dy 
-
\int_M\sum_j \cos t\la_j e_j^0(x)e_j^0(y) \, f(y)\, dy
\\
&=-\int_0^t
\bigl( \tfrac{\sin(t-s)P^0}{P^0}(V\cos (sP_V)f)\bigr)(x) \, ds
\\
&=-\int_0^t \int_M\int_M
\sum_j \tfrac{\sin(t-s)\la_j}{P^0}e_j^0(x)e_j^0(z)
V(z) \sum_k \cos s\tau_k e_{\tau_k}(z)e_{\tau_k}(y) f(y)\, dzdy ds.
\end{align*}
By \eqref{2.15} or \eqref{2.20} if we integrate this
against $\pi^{-1}\rho(\e t)\tfrac{\sin\la t}t$ we obtain
$\ola(P_V)f(x)-\ola(P^0)f(x)$.  Therefore, by
Lemma~\ref{triglemma} the kernel of $\ala(P_V)-\ala(P^0)$
is
\begin{multline}\label{2.21}
\bigl(\ala(P_V)-\ala(P^0)\bigr)(x,y)=
\\
\frac1\pi \sum_{j,k}\int_M\int_{-\infty}^\infty
\rho(\e t)\frac{\sin\la t}t 
\, m(\tau_k,\la_j) \, e_j^0(x)e_j^0(z) V(z)e_{\tau_k}(z)e_{\tau_k}(y) \, dzdt,
\end{multline}
where
\begin{equation}\label{2.22}
m(\tau,\mu)=
\begin{cases}
\frac{\cos t\tau-\cos t\mu}{\tau^2-\mu^2}, 
\quad \text{if } \, \tau\ne \mu
\\ \\
-\frac{t\sin t\tau}{2\tau}, \quad 
\text{if } \, \tau=\mu.
\end{cases}
\end{equation}
Thus, by \eqref{2.20}--\eqref{2.21} we have
\begin{multline}\label{2.20'}\tag{2.24$'$}
\bigl(\ala(P_V)-\ala(P^0)\bigr)(x,y)=
\\
\sum_{j,k}\int_M 
\frac{\ala(\tau_k)-\ala(\la_j)}{\tau_k^2-\la_j^2} e_j^0(x)e_j^0(z)V(z)e_{\tau_k}(z)
e_{\tau_k}(y) \, dz,
\end{multline}
if, by the second part of \eqref{2.22}
we interpret
\begin{equation}\label{2.23}
\frac{\ala(\tau)-\ala(\mu)}{\tau^2-\mu^2}
=\ala'(\tau)/2\tau, \quad \text{if } \, \, 
\tau=\mu.
\end{equation}

Thus, we would have \eqref{2.19} and consequently
Proposition~\ref{a.1} if we could prove the following:

\begin{proposition}\label{mainprop}
Let $V\in L^1(M)$ with $V^{-}\in
 {\mathcal K}(M)$, and $\ala(\tau)$ be defined as in \eqref{2.15}. Then we have
\begin{multline}\label{2.24}
\Bigl| \, \sum_{j,k}\int_M \int_M
\frac{\ala(\la_j)-\ala(\tau_k)}{\la_j^2-\tau_k^2} e_j^0(x)e_j^0(y)V(y)e_{\tau_k}(x)
e_{\tau_k}(y) \, dx dy\, \Bigr| 
\le C_V  \, 
\e^{-1}\la^{n-\frac32},
\end{multline}
for some constant $C_V$ depending on $V$.
\end{proposition}

As we shall see in the proof of Proposition~\ref{mainprop}, the constant $C_V$ depends on  $\|V\|_{L^1(M)}$. The proof will also use
heat kernel estimates involving $H_V$.  Steps like this
will contribute to the constant $C_V$ in 
\eqref{2.24}.

Note that the kernel in \eqref{2.20'} involves an amalgamation of the kernels of $\ala(P^0)$, $\ala(P_V)$ and
the resolvent kernels $(H_V-\mu^2)^{-1}$ and $(H^0-\mu^2)^{-1}$.  To prove \eqref{2.24} we shall attempt to
separate the contributions of the various components by using the following simple lemma.

\begin{lemma}\label{delta}  Let $I\subset \R_+$ and for
eigenvalues $\tau_k\in I$ assume that $\delta_{\tau_k}
\in [0,\delta]$.  Then if $m\in C^1(\R_+\times M)$
\begin{multline}\label{2.25}
\int_M \Bigl| \sum_{\tau_k\in I}m(\delta_{\tau_k},x) \,
a_k e_{\tau_k}(x)\Bigr| \, dx
\\
\le\Bigl(\, \|m(0, \, \cdot\, )\|_{L^2(M)}
+\int_0^\delta \bigl\| \tfrac\partial{\partial s}
m(s, \, \cdot\, )\bigr\|_{L^2(M)}\, ds
\, \Bigr) \times \bigl(\, \sum_{\tau_k\in I}
|a_k|^2\, \bigr)^{1/2}.
\end{multline}
\end{lemma}

\begin{proof}
We shall use the fact that
$m(\delta_{\tau_k},x)=m(0,x)+\int_0^{\delta}
\1_{[0,\delta_{\tau_k}]}(s)\, \tfrac\partial{\partial s}
m(s,x) \, ds$,
where $\1_{[0,\delta_{\tau_k}]}(s)$ is the indicator function of the the interval $[0,\delta_{\tau_k}]
\subset [0,\delta]$.
Therefore, by Minkowski's inequality,
the left side of \eqref{2.25} is dominated by
\begin{align*}
\int_M &\bigl|\, m(0,x)\cdot \sum_{\tau_k\in I}a_k
e_{\tau_k}(x)\, \bigr| \, dx
+\int_M \bigl| \, \sum_{\tau_k\in I}\int_0^\delta
\1_{[0,\delta_{\tau_k}]}(s)\tfrac\partial{\partial s}
m(s,x) a_k e_{\tau_k}(x)\, ds \, \bigr| \, dx
\\
&\le \int_M \bigl| \, m(0,x)\cdot \sum_{\tau_k\in I}a_k
e_{\tau_k}(x)\, \bigr| \, dx
+\int_0^\delta \Bigl(\, 
\int_M \bigl| \, \tfrac\partial{\partial s}m(s,x)\, 
\bigr| \cdot \bigl| \, 
\sum_{\tau_k\in I}a_k
e_{\tau_k}(x)
\, \bigr| \, dx
\, \Bigr)\, ds
\\
&\le 
\|m(0,\, \cdot\, )\|_2 \cdot 
\|\sum_{\tau_k\in I}a_ke_{\tau_k}\|_2
+\int_0^\delta\bigl( \, 
\bigl\|\tfrac\partial{\partial s}
m(s,\, \cdot \, )\|_2 \cdot \bigl\| \1_{[0,\delta_{\tau_k}]}(s)a_ke_{\tau_k}\, \bigr\|_2\, \bigr)\, ds
\\
&=\|m(0,\, \cdot \, )\|_2 \cdot (\sum_{\tau_k\in I}
|a_k|^2)^{1/2}
+\int_0^\delta \bigl\|\tfrac\partial{\partial s}m(s, \, \cdot\, )\|_2\cdot \bigl(\sum_{\tau_k\in I}
|\1_{[0,\delta_{\tau_k}]}(s)a_k|^2\bigr)^{1/2} \, ds
\\
& \le\Bigl(\, \|m(0, \, \cdot\, )\|_{L^2(M)}
+\int_0^\delta \bigl\| \tfrac\partial{\partial s}
m(s, \, \cdot\, )\bigr\|_{L^2(M)}\, ds
\, \Bigr) \times \bigl(\, \sum_{\tau_k\in I}
|a_k|^2\, \bigr)^{1/2},
\end{align*}
as desired.
\end{proof}

Next, recall that we mentioned that the kernel in \eqref{2.24} is a juxtaposition of the kernels of
$\ala(P^0)$ as well as resolvent-type kernels.  To handle the former, we shall appeal to the following straightforward result.

\begin{lemma}\label{ala}
Let $\ala(P^0)$ be defined by \eqref{2.15} and
the analog of \eqref{2.20} involving $P^0$.  Then
the kernel of 
$(P^0)^\mu\ala(P^0)$, $\mu=0,1,2,\dots$ satisfies
\begin{equation}\label{2.26}
\bigl((P^0)^\mu\, \ala(P^0)\bigr)(x,y)=
\sum_j \la_j^\mu \ala(\la_j)e_j^0(x)e_j^0(y)=
O(\la^{n+\mu}),
\end{equation}
and, moreover,
\begin{equation}\label{2.27}
\bigl\| \bigl((P^0)^\mu\, \ala(P^0)\bigr(\, \cdot\, ,y )\bigr\|_{L^2(M)}
=O(\la^{n/2+\mu}),
\end{equation}
where the implicit constants are independent of $\e$.
\end{lemma}

The proof of the lemma is very simple.  
First, by the pointwise Weyl formula of Avakumovi\'{c}~\cite{Avakumovic}, Levitan~\cite{Levitan}
and H\"ormander~\cite{HSpec}  (see also \cite{SoggeHangzhou}),
\begin{equation}\label{freespec}\sum_{\la_j\in [\ell,\ell+1]}|e^0_j(x)e^0_j(y)|=
O(\ell^{n-1}), \quad \ell \in {\mathbb N}.
\end{equation}
If we use this and \eqref{2.16}, we obtain \eqref{2.26}.  To prove the other inequality, \eqref{2.27}, we note that, by orthogonality
$$
\bigl\|(P^0)^\mu\ala(P^0)(\, \cdot\, ,y)\bigr\|_{L^2(M)}^2
=\sum_j \la_j^{2\mu}\bigl(\ala(\la_j)\bigr)^2 |e_j^0(y)|^2
=O(\la^{n+2\mu}),
$$
by this argument, which is \eqref{2.27}.

\medskip

To deal with the contributions of resolvent type operators in the mixture \eqref{2.24} we shall need a couple more results.  The first is bounds for cutoff resolvent operators for the free operator $H^0$.

\begin{lemma}\label{freeres}
Fix $\eta\in C^\infty(\R_+)$ satisfying
$\eta(s)=0$ on $s\le 2$ and $\eta(s)=1$, $s>4$.  Then if we set for $\tau 
\gg 1$
\begin{equation}\label{2.29}
R_\tau(x,y)=\sum_j \frac{\eta(\la_j/\tau)}{
\la_j^2-\tau^2} \, e_j^0(x)e_j^0(y).
\end{equation}
we have
\begin{equation}\label{2.30}
|R_\tau(x,y)|\le C_N \tau^{n-2}
h_n\bigl(\tau \, d_g(x,y)\bigr) \, 
\bigl(1+\tau \, d_g(x,y)\bigr)^{-N},
\end{equation}
for any $N=1,2,3,\dots$, if  $h_n(r)$ is as in \eqref{1.4}.
The constant
$C_N$ depends on $N$, $(M,g)$ and finitely many
derivatives of $\eta$.
\end{lemma}

Here we are abusing the notation a bit.  In 
 \eqref{2.30} we mean that the inequality holds near
the diagonal (so that $d_g(x,y)$ is well-defined) and that outside of this neighborhood of the diagonal $R_\tau(x,y)$ is $O(\tau^{-N})$ for all $N$.  We shall
state certain inequalities in this manner in what follows.

To verify \eqref{2.30}, we note that the integral
operator $R_\tau$ arising from the kernel
$R_\tau(x,y)$ is
$$\tau^{-2} m(P^0/\tau),$$
where
$$m(\mu)=\frac{\eta(|\mu|)}{\mu^2-1}.$$
Thus, $m$ is a symbol of order -2, i.e.,
$$\partial_\mu^j m(\mu)=O((1+\mu)^{-2-j}), \, j=0,1,2,\dots.$$
As a result, one can use the arguments in \cite[\S 4.3]{SFIO2} to see that \eqref{2.30} is valid.  Indeed, modulo lower order terms, $R_\tau(x,y)$ equals
$$(2\pi)^{-n}\int_{\Rn}
\tau^{-2}\frac{\eta(|\xi|/\tau)}{(|\xi|/\tau)^2-1}
\, e^{id_g(x,y)\xi_1} \, d\xi,$$
near the diagonal, which satisfies the bounds in \eqref{2.30}, while outside of a fixed neighborhood of the diagonal $R_\tau(x,y)=O(\tau^{-N})$ for all $N$.

\medskip

We also need bounds for the kernels of $(H_V)^{-j}$.

\begin{lemma}\label{heat}
Let $(H_V)^{-j}(x,y)=\sum_k \tau_k^{-2j} e_{\tau_k}(x)e_{\tau_k}(y)$ be the kernel of $(H_V)^{-j}$, $j=1,2,\dots$. 
Then if $h_n(r)$ is
as in \eqref{1.4}
\begin{equation}\label{2.31}
\bigl(H_V\bigr)^{-1}(x,y)\lesssim 
\begin{cases}h_n(d_g(x,y)), \quad \text{if }\, \, 
d_g(x,y)\le \textrm{Inj }(M)/2,
\\ 1, \quad \text{otherwise}.
\end{cases}
\end{equation}
Furthermore, if $n\ge 5$ and $j <n/2$, $j\in {\mathbb N}$ we have 
\begin{equation}\label{2.32}
\bigl(H_V\bigr)^{-j}(x,y)\lesssim 
\begin{cases}
\bigl(d_g(x,y)\bigr)^{-n+2j}, \quad \text{if }\, \, 
d_g(x,y)\le \textrm{Inj }(M)/2,
\\
1 \quad \text{otherwise}.
\end{cases}
\end{equation}
\end{lemma}

To prove \eqref{2.31} or \eqref{2.32}, we note that
\begin{equation}\label{2.33}\bigl(H_V\bigr)^{-j}(x,y)=\int_0^\infty 
t^{j-1}  \, \bigl(e^{-tH_V}\bigr)(x,y)\, dt.
\end{equation}
We then use the heat kernel estimates of Li and Yau~\cite{LiYau} ($V\in C^\infty$) and 
Sturm~\cite[(4.14) Corollary]{Sturm}
($V\in {\mathcal K}(M)$), which say that for $0<t\le 1$
there is a uniform constant $c=c_{M,V}>0$ so that
\begin{equation}\label{2.34}
(e^{-tH_V}\bigr)(x,y)
\lesssim 
\begin{cases} t^{-n/2} \exp(-c(d_g(x,y))^2/t), \quad 
\quad \text{if }\, \, 
d_g(x,y)\le \textrm{Inj }(M)/2,
\\
1 \quad \text{otherwise}.
\end{cases}
\end{equation}
here we do not need $V^{+} \in {\mathcal K}(M)$, since by $Feynman$-$Kac \,\,formula$, $e^{-tH_V}|f|$ is monotone decreasing as $V^{+}$ increases, so 
$V^{+}\in L^1(M)$ will not affect the bound in \eqref{2.34}.

As a consequence of \eqref{2.34},
we have for $0<t\leq 1$
$$\int_M |(e^{-tH_V}\bigr)(x,y)|^2 dy \lesssim t^{-\frac n2}.
$$
By Schwarz's inequality, we have $\|e^{-tH_V}\|_{L^2\rightarrow L^\infty}\lesssim t^{-\frac n4}$. If we consider the kernels of the dyadic spectral projection operators
\begin{equation}\label{2.35}
\tilde\chi_\la^V(x,y)=\sum_{\tau_k\in [\la,2\la)}e_{\tau_k}(x)e_{\tau_k}(y),
\end{equation}
for $H_V$, then, by the spectral theorem, we have $$\|\tilde\chi_\la^V\|_{L^2\rightarrow L^\infty} \lesssim \|e^{-\la^{-2}H_V}\|_{L^2\rightarrow L^\infty}\lesssim \la^{\frac n2},$$
which, along with the Cauchy-Schwarz inequality, implies
\begin{equation}\label{2.36}
\sup_{x,y\in M}\bigl|\sum_{\tau_k\in [\la,2\la)}e_{\tau_k}(x)e_{\tau_k}(y)\bigr|
\le \sup_{x\in M}\sum_{\tau_k\in [\la,2\la)}|e_{\tau_k}(x)|^2 = \bigl\|\tilde\chi_\la^V \bigr\|_{L^2\rightarrow L^\infty}^2
 \lesssim \la^n.
\end{equation}

Since the eigenvalues of $H^V$ are all $\ge 1$, by \eqref{2.36} we have
\begin{equation}\label{2.37}
(e^{-tH_V}\bigr)(x,y)\lesssim e^{-t/2}, \quad t>1.
\end{equation}
If we use \eqref{2.34}, \eqref{2.37} along with \eqref{2.33}, we  obtain \eqref{2.31} and \eqref{2.32}.

\newsection{Proof of the universal Weyl law involving singular potentials}

To prove Proposition~\ref{mainprop}, which, as noted, implies our main result, Theorem~\ref{thm1.1}, we shall split things into three different cases that require slightly different arguments.  Specifically, we shall first handle the contribution of frequencies $\tau_k$ which are comparable to $\lambda$, and then those that are relatively small followed by ones that are  relatively large.

\subsection*{Handling the contribution of comparable frequencies}

In this subsection we shall handle frequencies $\tau_k$ which are comparable to $\lambda$, which one would expect to be
the main contribution to the Weyl error term in \eqref{1.12}.  Specifically, we shall prove the following.

\begin{proposition}\label{compprop}
Let $V\in L^1(M)$ with $V^{-}\in
 {\mathcal K}(M)$, and $\ala(\tau)$ be defined as in \eqref{2.15} with $\e=\e(\la)$ satisfying \eqref{a.1}. Then we have
\begin{multline}\label{3.1}
\Bigl| \, \sum_{j}
\sum_{\{k: \, \tau_k\in [\lambda/2,10\lambda]\}}\int_M \int_M
\frac{\ala(\la_j)-\ala(\tau_k)}{\la_j^2-\tau_k^2} e_j^0(x)e_j^0(y)V(y)e_{\tau_k}(x)
e_{\tau_k}(y) \, dx dy\, \Bigr| 
\\
\le C_V \e^{-1} \, 
\la^{n-\frac32},
\end{multline}
for some constant $C_V$ depending on $V$.
\end{proposition}

To prove Proposition~\ref{compprop}, let us fix a Littlewood-Paley bump function
$\beta\in C^\infty_0((1/2,2))$ satisfying
$$
\sum_{\ell=-\infty}^\infty \beta(2^{-\ell} s)=1, \quad s>0.$$
Let $\ell_0\le 0$ be the largest integer such that $2^{\ell_0}\le \e$, and set
$$\beta_{\ell_0}(s)=\sum_{\ell\le \ell_0} \beta(2^{-\ell}|s|)
\in C^\infty_0((-2,2)),$$
and 
$$\tilde \beta(s)=s^{-1}\beta(|s|)\in C^\infty_0
\bigl(\{|s|\in (1/2,2)\}\bigr).$$

We then write for $\la/2\le \tau\le 10\la$
\begin{align} \label{3.2}
K_\tau(x,y)&=\sum_j \frac{\ala(\la_j)-\ala(\tau)}{\la_j^2-\tau^2}e_j^0(x)e_j^0(y)
\\
&= \sum_j \frac{\ala(\la_j)-\ala(\tau)}{\la_j-\tau}\frac{\beta_{\ell_0}(\la_j-\tau)}{\la_j+\tau}\, e_j^0(x)e_j^0(y) \notag
\\
&\, \, \, \,  \, \, +\sum_{\{\ell\in {\mathbb Z}: \, \e<2^\ell \le \la/100\}} \Bigl(\, \sum_j
\frac{2^{-\ell}\tilde\beta(2^{-\ell}(\la_j-\tau))}
{\la_j+\tau}
\, (\ala(\la_j)-\ala(\tau)) \, e^0_j(x)e^0_j(y)\, \Bigr)
\notag
\\
&\, \, \, \,  \, \, + 
\sum_j \Bigl(\frac{\sum_{\{\ell\in {\mathbb Z}: \, 2^\ell > \la/100\}}\beta(2^{-\ell}(\la_j-\tau))}{\la_j^2-\tau^2}
\Bigr)\bigl(\ala(\la_j)-\ala(\tau)\bigr) \,
e_j^0(x)e_j^0(y).  
\notag
\end{align}

Next, let 
$$K_{\tau,\ell_0}(x,y)= \sum_j \frac{\ala(\la_j)-\ala(\tau)}{\la_j-\tau}\frac{\beta_{\ell_0}(\la_j-\tau)}{\la_j+\tau} \, e_j^0(x)e_j^0(y),$$
$$R_{\tau,\ell}(x,y)= \sum_j
\frac{2^{-\ell}\tilde\beta(2^{-\ell}(\la_j-\tau))}
{\la_j+\tau}
\,  e^0_j(x)e^0_j(y),
\quad \text{if } \, \e<2^\ell \le \lambda/100,$$
and
$$R_{\tau,\infty}(x,y)=
\sum_j \Bigl(\frac{\sum_{\{\ell\in {\mathbb Z}: \, 2^\ell > \la/100\}}\beta(2^{-\ell}(\la_j-\tau))}{\la_j^2-\tau^2}
\Bigr) \,
e_j^0(x)e_j^0(y).$$
Also, for $\e<2^\ell \le \la/100$ let
\begin{align*}
K^-_{\tau,\ell}(x,y)&=\sum_j
\frac{2^{-\ell}\tilde \beta(2^{-\ell}(\la_j-\tau))}{\la_j+\tau} \bigl(\ala(\la_j)-1\bigr) \, e^0_j(x)e^0_j(y)
\\ 
K^+_{\tau,\ell}(x,y)&=\sum_j
\frac{2^{-\ell}\tilde \beta(2^{-\ell}(\la_j-\tau))}{\la_j+\tau} \ala(\la_j) \, e^0_j(x)e^0_j(y),
\end{align*}
and, finally,
\begin{align*}
K^-_{\tau,\infty}(x,y)&=
\sum_j \Bigl(\frac{\sum_{\{\ell\in {\mathbb Z}: \, 2^\ell > \la/100\}}\beta(2^{-\ell}(\la_j-\tau))}{\la_j^2-\tau^2}
\Bigr) \bigl(\ala(\la_j)-1\bigr) \,
e_j^0(x)e_j^0(y)
\\
K^+_{\tau,\infty}(x,y)&=
\sum_j \Bigl(\frac{\sum_{\{\ell\in {\mathbb Z}: \, 2^\ell > \la/100\}}\beta(2^{-\ell}(\la_j-\tau))}{\la_j^2-\tau^2}
\Bigr) \ala(\la_j) \,
e_j^0(x)e_j^0(y).
\end{align*}

If $K_\tau$ is as in \eqref{3.2}, our current task,
\eqref{3.1}, is to show that
\begin{equation}\label{3.1'}\tag{3.1$'$}
\Bigl| \sum_{\tau_k \in [\la/2,10\la]}\iint
K_{\tau_k}(x,y)e_{\tau_k}(x)e_{\tau_k}(y) \, V(y)\, dxdy\Bigr|
\le C_V \e^{-1} \la^{n-\frac32}.
\end{equation}

To prove this, we note that we can write
\begin{multline}\label{3.3}
K_\tau(x,y)=K_{\tau,\ell_0}(x,y)
+ \sum_{\{\ell \in {\mathbb Z}: \, \e<2^\ell \le \la/100\}}
K^-_{\tau,\ell}(x,y)+K^-_{\tau,\infty}(x,y)
\\
+ \sum_{\{\ell \in {\mathbb Z}: \, \e<2^\ell \le \la/100\}}
R_{\tau,\ell}(x,y)\bigl(1-\ala(\tau)\bigr)
+R_{\tau,\infty}(x,y) \bigl(1-\ala(\tau)\bigr),
\end{multline}
or
\begin{multline}\label{3.4}
K_\tau(x,y)=K_{\tau,\ell_0}(x,y)
+ \sum_{\{\ell \in {\mathbb Z}: \, \e<2^\ell \le \la/100\}}
K^+_{\tau,\ell}(x,y)+K^+_{\tau,\infty}(x,y)
\\
- \sum_{\{\ell \in {\mathbb Z}: \, \e<2^\ell \le \la/100\}}
R_{\tau,\ell}(x,y)\ala(\tau)
-R_{\tau,\infty}(x,y) \ala(\tau),
\end{multline}
We shall use \eqref{3.3} to handle the summands in 
\eqref{3.1'} with
$\tau=\tau_k\in [\la/2,\la]$ and
\eqref{3.4} to handle those with $\tau=\tau_k\in (\la,10\la]$.

For $\ell \in {\mathbb Z}$ with $\e<2^\ell \le \la/100$, let for $j=0,1,2,\dots$
\begin{equation}\label{3.5}
I^-_{\ell,j}=\bigl(\la-(j+1)2^\ell, \, \la
-j2^\ell\bigr]
\quad \text{and} \quad
I^+_{\ell,j}=\bigl(\la+j2^\ell, \, 
\la+(j+1)2^\ell \, \bigr].
\end{equation}
Then to use the $\delta_\tau$--Lemma (Lemma~\ref{delta}), we shall use the following result whose proof we
momentarily postpone.

\begin{lemma}\label{lemma3.2}
If $\ell \in {\mathbb Z}$, $\e<2^\ell\le \la/100$,
and $j=0,1,2,\dots$, we have for each $N\in {\mathbb N}$
\begin{multline}\label{3.6}
\|K^{\pm}_{\tau,\ell}(\, \cdot\, ,y )\|_{L^2(M)}, \, \,
\|2^\ell \tfrac\partial{\partial\tau}K^{\pm}_{\tau,\ell}(\, \cdot\, ,y )\|_{L^2(M)}
\\
\lesssim \e^{-1/2}\la^{\frac{n-1}2-1}2^{-\ell/2}(1+j)^{-N},
\quad \tau\in I_{\ell,j}^\pm \cap [\la/2,10\la].
\end{multline}
Also, 
\begin{multline}\label{3.7}
\|K_{\tau,\ell_0}(\, \cdot\, ,y )\|_{L^2(M)}, \, \, 
\|\e\cdot\tfrac\partial{\partial\tau}K_{\tau,\ell_0}(\, \cdot\, ,y )\|_{L^2(M)} 
\\
\lesssim \e^{-1}
\la^{\frac{n-1}2-1}(1+j)^{-N}, \quad 
\tau\in I_{\ell_0,j}^\pm \cap [\la/2,10\la],
\end{multline}
\begin{equation}\label{3.8}
\|K^+_{\tau,\infty}( \, \cdot\, ,y )\|_{L^2(M)}, \, \,
\|\la \tfrac\partial{\partial \tau}K^+_{\tau,\infty}( \, \cdot\, ,y)\|_{L^2(M)}
\lesssim \la^{\frac{n}2-2}, \quad \tau\in [\la,10\la],
\end{equation}
and we can write
$$K^-_{\tau,\infty}(x,y)=\tilde K^-_{\tau,\infty}(x,y)
+H^-_{\tau,\infty}(x,y),$$
where for $\tau\in [\la/2,\la]$
\begin{equation}\label{3.9}\begin{split}
\|\tilde K^-_{\tau,\infty}(\, \cdot \, ,y )\|_{L^2(M)}
, \, \, \|\la \tfrac\partial{\partial \tau}
\tilde K^-_{\tau,\infty}(\, \cdot \, ,y )\|_{L^2(M)} \lesssim \la^{\frac{n}2-2}
\\ \\
|H^-_{\tau,\infty}(x,y)|\lesssim
\la^{n-2}h_n(\la d_g(x-y))(1+\la d_g(x,y))^{-N}, 
\end{split}
\end{equation}
where $h_n$ is as in \eqref{1.4}.
Finally,
we also have for $\e< 2^\ell \le \la/100$ and $\tau\in [\la/2,10\la]$
\begin{equation}\label{3.10}
\|R_{\tau,\ell}(\, \cdot\, ,y )\|_{L^2(M)}, \, \, 
\|2^\ell \tfrac\partial{\partial \tau}R_{\tau,\ell}(\, \cdot\, ,y)\|_{L^2(M)}
\lesssim \e^{-1/2}\la^{\frac{n-1}2-1} 2^{-\ell/2},
\end{equation}
and
\begin{equation}\label{3.11}
|R_{\tau,\infty}(x,y)|\lesssim \la^{n-2}h_n(\la d_g(x,y))
\, 
(1+\la d_g(x,y))^{-N}.
\end{equation}
\end{lemma}

As before, we are abusing notation a bit.  First, in \eqref{3.6}
we mean that if $K_{\tau,\ell}$ equals $K^+_{\tau,\ell}$ or
$K^-_{\tau,\ell}$ then the bounds in \eqref{3.6} for
$\tau$ in $I_{\ell,j}^+ \cap [\la,10\la]$ or
$I_{\ell,j}^- \cap [\la/2,\la]$, respectively.  Also,
in both the second inequality in \eqref{3.9} and in \eqref{3.11} we mean that the kernels satisfy the bounds when $x$ is sufficiently close to $y$ (so that $d_g(x,y)$ is well-defined) and that they are $O(\la^{-N})$ away from the diagonal.

We shall also need the following lemma for the proof of Proposition~\ref{compprop}.
\begin{lemma}\label{lemma3.3}
Let $I=[a_0, a_0+\gamma]$ be an interval of length $\gamma \le \la$, and assume that for any fixed $\tau \in I \cap [\la/2,10\la]$, $w_\tau(x,y)\in C^1(\mathbb{R}\times M\times M)$ satisfies
\begin{equation}\label{b.1}
\|w_{\tau}(\, \cdot\, ,y )\|_{L^2(M)}, \, \,
\|\gamma \tfrac\partial{\partial\tau}w_{\tau}(\, \cdot\, ,y )\|_{L^2(M)}
\le L,
\end{equation}
for some constant $L$.
Then if $\beta\in C^\infty(\mathbb{R})$ and $V\in L^1(M)$, we have
\begin{multline}\label{b.2}
\Bigl| \sum_{\tau_k\in I\cap[\la/2,10\la]}
\iint w_{\tau_k}(x,y)\beta(\tau_k)e_{\tau_k}(x)e_{\tau_k}(y)
V(y)\, dydx\Bigr| \\
\lesssim \sup_{\tau\in I\cap[\la/2,10\la]}|\beta(\tau)|\cdot  L\,\la^{n/2},
\end{multline}
\end{lemma}
\begin{proof}
For any fixed $y\in M$, by applying Lemma~\ref{delta} with
$\delta=\gamma$, $m(\tau, x)=w_{\tau+a}(x,y)$ and $a_k=\beta(\tau_k)e_{\tau_k}(y)$, we have
\begin{align*}
\Bigl| &\sum_{\tau_k\in I\cap[\la/2,10\la]}
\iint w_{\tau_k}(x,y)\beta(\tau_k)e_{\tau_k}(x)e_{\tau_k}(y)
V(y)\, dydx\Bigr|
\\
&\le \|V\|_{L^1}
\cdot \sup_y \Bigl\|
\sum_{\tau_k\in I\cap[\la/2,10\la]}
 w_{\tau_k}(x,y)e_{\tau_k}(x)e_{\tau_k}(y)
\Bigr\|_{L^1(dx)}
\notag
\\
&\le \|V\|_{L^1} \cdot \sup_y
\Bigl( \|w_{a_0}(\, \cdot \, ,y)\|_{L^2(M)}+\int_{0}^\gamma\bigl\|\tfrac\partial{\partial \tau}w_{s+a_0}(\, \cdot \, y)\|_{L^2(M)} \, ds
\Bigr) \notag
\\
&\qquad \qquad \qquad \qquad \qquad \times \bigl(\sum_{\tau_k\in I\cap [\la/2,10\la]}
|\beta(\tau_k)e_{\tau_k}(y)|^2\bigr)^{1/2}
\notag
\\
&\lesssim \sup_{\tau\in I\cap[\la/2,10\la]}|\beta(\tau)|\cdot  L\,\la^{n/2}.
\notag
\end{align*}
In the second to last inequality we used \eqref{b.1} and the fact that, 
by \eqref{2.36}, 
\begin{equation}\label{3.13}
\sum_{\tau_k\in [\la/2,10\la]} |e_{\tau_k}(y)|^2 \lesssim
\la^{n}, \quad \la\ge 1.
\end{equation}
\end{proof}

\begin{proof}[Proof of Proposition~\ref{compprop}]
First, if we apply Lemma~\ref{lemma3.3} with $w_\tau(x,y)=K^{\pm}_{\tau,\ell}(x,y )$, $\gamma=2^\ell$ and $\beta(\tau)\equiv 1$, by \eqref{3.6} 
\begin{equation}\label{b.4}
\Bigl| \sum_{\tau_k\in I^\pm_{\ell,j}\cap (\la,10\la]}\iint K^\pm_{\tau_k,\ell}(x,y)e_{\tau_k}(x)e_{\tau_k}(y)
V(y)\, dydx\Bigr|  \lesssim \e^{-1/2} \la^{n-\frac 32}2^{-\ell/2} \cdot (1+j)^{-N}
\end{equation}
If we sum over $j=0,1,2,\dots$, we see that for $\e< 2^\ell \le \la/100$, \eqref{b.4} yields
\begin{multline}\label{3.14}
\Big| \sum_{\la <\tau_k\le 10\la}
\iint K^+_{\tau_k,\ell}(x,y) e_{\tau_k}(x)e_{\tau_k}(y) 
\, V(y)\, dxdy\Bigr|
\\
+\Big| \sum_{\la/2 \le \tau_k\le\la}
\iint K^-_{\tau_k,\ell}(x,y) e_{\tau_k}(x)e_{\tau_k}(y) 
\, V(y)\, dxdy\Bigr|
\lesssim \e^{-1/2} \la^{n-\frac 32}2^{-\ell/2}.
\end{multline}
If we take $w_\tau(x,y)=K_{\tau,\ell_0}(x,y )$, $\gamma=\e$ and $\beta(\tau)\equiv 1$ in Lemma~\ref{lemma3.3}, this argument
also gives
\begin{multline}\label{3.15}
\Bigl| \sum_{\la/2 \le \tau_k\le \la}
\iint K_{\tau_k,\ell_0}(x,y) e_{\tau_k}(x)e_{\tau_k}(y) \, 
V(y)\, dx dy\Bigr|
\\
+ \Bigl| \sum_{\la < \tau_k\le 10\la}
\iint K_{\tau_k,\ell_0}(x,y) e_{\tau_k}(x)e_{\tau_k}(y) \, 
V(y)\, dx dy\Bigr|
\lesssim \e^{-1}\la^{n-\frac32}.
\end{multline}
Similarly, if we use Lemma~\ref{lemma3.3}
along with \eqref{3.8} we find that
\begin{multline}\label{3.16}
\Bigl| \sum_{\la< \tau_k\le 10\la}
\iint K_{\tau_k,\infty}^+(x,y) e_{\tau_k}(x)e_{\tau_k}(y) \, 
V(y)\, dx dy\Bigr|
\\
\lesssim \la^{\frac{n}2-2} \|V\|_{L^1} \bigl(\sum_{\tau_k 
\in [\la/2,10\la]} |e_{\tau_k}(y)|^2\bigr)^{1/2}
\lesssim \la^{n-2},
\end{multline}
using \eqref{3.13} for the last inequality.

Next, since $R_{\tau,\ell}$ enjoys the bounds in \eqref{3.10}, we can use Lemma~\ref{lemma3.3} with $w_\tau(x,y)=R_{\tau,\ell}(x,y )$, $\gamma=2^\ell$ and $\beta(\tau)=\ala(\tau)$ 
to see that for $\e< 2^\ell \le \la/100$ we
have
\begin{align*}
\Bigl| \sum_{\tau_k\in I^+_{\ell,j}\cap (\la,10\la]}
&\iint R_{\tau_k,\ell}(x,y)\ala(\tau_k)e_{\tau_k}(x)
e_{\tau_k}(y) V(y)\, dxdy\Bigr|
\\
&\lesssim \|V\|_{L^1}\cdot 2^{-\ell/2}\la^{\frac{n-1}2-1}
\sup_y\bigl(\sum_{\tau_k\in I^+_{\ell,j}\cap
(\la,10\la]} |\ala(\tau_k)e_{\tau_k}(y)|^2\bigr)^{1/2} 
\\
&\lesssim \e^{-1/2}\la^{n-\frac32}2^{-\ell/2}\|V\|_{L^1}\cdot (1+j)^{-N},
\end{align*}
since $\ala(\tau_k)=O((1+j)^{-N})$ if $\tau_k
\in I^+_{\ell,j}$.  Summing over this bound over $j$
of course yields
\begin{equation}\label{3.17}
\Bigl|\sum_{\la<\tau_k\le 10\la}
\iint R_{\tau_k,\ell}(x,y)
\ala(\tau_k)e_{\tau_k}(x)e_{\tau_k}(y) V(y)\, dxdy\Bigr|
\lesssim \e^{-1/2}\la^{n-\frac32}2^{-\ell/2}.
\end{equation}
The same argument gives
\begin{multline}\label{3.18}
\Bigl|\sum_{\la/2\le\tau_k\le \la}
\iint R_{\tau_k,\ell}(x,y)
\bigl(1-\ala(\tau_k)\bigr)e_{\tau_k}(x)e_{\tau_k}(y) V(y)\, dxdy\Bigr|
 \lesssim \e^{-1/2}\la^{n-\frac32}2^{-\ell/2}.
\end{multline}
Also, by \eqref{3.11} we have
$$\sup_y \int \sup_{\la/2\le \tau\le 10\la}|R_{\tau,\infty}(x,y)|\, dx\lesssim \la^{-2},
$$
and since \eqref{3.13} yields
$\sum_{\tau_k\le 10\la} |e_{\tau_k}(x)e_{\tau_k}(y)|
\lesssim \la^n$, we have
\begin{equation}\label{3.19}\begin{split}
\Bigl|\sum_{\tau_k\in (\la,10\la]}
\iint R_{\tau_k,\infty}
(x,y) \ala(\tau_k)e_{\tau_k}(x)e_{\tau_k}(y) V(y)
\, dxdy\Bigr|\lesssim \la^{n-2}\|V\|_{L^1}
\\
\Bigl|\sum_{\tau_k\in [\la/2,\la]}\iint
R_{\tau_k,\infty}
(x,y) \bigl(1-\ala(\tau_k)\bigr)e_{\tau_k}(x)e_{\tau_k}(y) V(y)
\, dxdy\Bigr|\lesssim \la^{n-2}\|V\|_{L^1}
\end{split}
\end{equation}

If $H^-_{\tau,\infty}$ is as in \eqref{3.9} this argument
also gives us
$$\Bigl|\sum_{\la/2\le \tau_k\le \la} \iint
H^-_{\tau_k,\infty}(x,y) e_{\tau_k}(x)e_{\tau_k}(y)
\, V(y)\, dxdy\Bigr|\lesssim
\la^{n-2}\|V\|_{L^1},$$
while the proof of \eqref{3.16} along with the first part of \eqref{3.9} yields
$$\Bigl| \sum_{\la/2 \le \tau_k\le \la}
\iint \tilde K_{\tau_k,\infty}^-(x,y) e_{\tau_k}(x)e_{\tau_k}(y) \, 
V(y)\, dx dy\Bigr|
\lesssim \la^{n-2}\|V\|_{L^1}.$$
Since $K_{\tau,\infty}=\tilde K_{\tau,\infty}+
H^-_{\tau,\infty}$, we deduce
\begin{equation}\label{3.20}
\Bigl| \sum_{\la/2 \le \tau_k\le \la}
\iint  K_{\tau_k,\infty}^-(x,y) e_{\tau_k}(x)e_{\tau_k}(y) \, 
V(y)\, dx dy\Bigr|
\lesssim \la^{n-2}\|V\|_{L^1}.
\end{equation}

We now have assembled all the ingredients for the proof
of \eqref{3.1'}.  If we use \eqref{3.14}, \eqref{3.15},
\eqref{3.18}, \eqref{3.19} and \eqref{3.20} along with
\eqref{3.3}, we conclude that the analog of \eqref{3.1'}
must be valid where the sum is taken over
$\tau_k\in [\la/2,\la]$. 
We similarly obtain the analog of \eqref{3.1'} where the sum is taken over $\tau_k\in (\la,10\la]$ from \eqref{3.4} along with
\eqref{3.14}, \eqref{3.15}, \eqref{3.16}, \eqref{3.17} and \eqref{3.19}.

From this, we deduce that \eqref{3.1'} must be valid,
which finishes the proof of Proposition~\ref{compprop}.
\end{proof}

To finish the present task we need to prove Lemma~\ref{3.2}.

\begin{proof}[Proof of Lemma~\ref{lemma3.2}] 
Since, as defined in \eqref{1.9}, $\{e^0_j\}_{j=1}^\infty$ are orthonormal bases for $L^2(M)$. By $L^2$ orthogonality, we have
\begin{equation}\label{b.5}
\Bigl(\int_M\Bigl|\sum_i a(\la_i,\tau)e_i^0(x)e_i^0(y)\Bigr|^2 dx\Bigr)^{1/2}=\bigl(\sum_{i} |a(\la_i,\tau)e^0_i(y)|^2\bigr)^{1/2},\,\,\forall\,\,\,y\in M.
\end{equation}

To prove the first inequality we note that if $\tau\in I^\pm_{\ell,j}\cap[\la/2,10\la]$ with $\beta(2^{-\ell}(\la_i-\tau))\ne 0$, then
$|\la_i-\tau|\le 2^{\ell+1}$, $\la_i,\tau\approx \la$, and, in this case,
we also have $\ala(\la_i)-1=O((1+|j|)^{-N})$ if $\tau\in I^-_{\ell,j}$ and $\ala(\la_i)=O((1+|j|)^{-N})$ if
$\tau\in I^+_{\ell,j}$.  Therefore, by \eqref{b.5} and \eqref{freespec}, we have for $\e<2^\ell \le \la/100$
\begin{align*}
\bigl\|K^\pm_{\tau,\ell}(\, \cdot \, ,y)\bigr\|_{L^2(M)}&\lesssim
(1+|j|)^{-N}2^{-\ell}\la^{-1} \, \bigl(\sum_{\{i: \, |\la_i-\tau|\le 2^{\ell+1}\}} |e^0_i(y)|^2\bigr)^{1/2}
\\
&\lesssim (1+|j|)^{-N}2^{-\ell}\la^{-1}\bigl(\sum_{\{\mu\in \{{\mathbb N}: \, |\mu-\tau|\le 2^{\ell+1}\}}
\mu^{n-1}\bigr)^{1/2}
\\
&\le \begin{cases}(1+|j|)^{-N}2^{-\ell}\la^{\frac{n-1}2-1}, \,\,\,\text{if}\,\,\, 2^{\ell}< 1\\ 
(1+|j|)^{-N}2^{-\ell/2}\la^{\frac{n-1}2-1}, \,\,\,\text{if}\,\,\, 2^{\ell}\ge 1,
\end{cases}
\end{align*}
which gives us the first part of \eqref{3.6} since $\max\{2^{-\ell}, 2^{-\ell/2}\}\le \e^{-1/2}2^{-\ell/2}$.  In the second inequality, we used \eqref{freespec}.  The other inequality
in \eqref{3.6} follows from this argument since
$$\frac\partial{\partial \tau}
\frac{\tilde \beta(2^{-\ell}(\la_i-\tau))}{\la_i+\tau}=O(2^{-\ell}\la^{-1}),
$$
due to the fact that we are assuming that $\e<2^\ell \le \la/100$.

This argument in the proof of  \eqref{3.6} also gives us \eqref{3.7} if we use the fact that
$\tau\to (\ala(\tau)-\ala(\mu))/(\tau^2-\mu^2)$ is smooth if we define it as in \eqref{2.23} when
$\tau=\mu$ (which is consistent with \eqref{2.20'}) and use the fact that
$$\partial_\tau^k \bigl(\beta_{\ell_0}(\la_i-\tau) (\ala(\la_i)-\ala(\tau))/(\la_i-\tau)\bigr)=O(\e^{-1-k}(1+|j|)^{-N}), \, \, k=0,1,
\, \, \tau\in I_{\ell_0,j}^\pm,$$
and the fact that, if this expression is nonzero, we must have $|\la_i-\tau|\le 2\e$.

To prove \eqref{3.8} we use the fact that for $k=0,1$ we have for $\tau\in (\la,10\la]$
$$\Bigl| \, \Bigl(\frac\partial{\partial \tau}\Bigr)^k
\Bigl(\frac{\sum_{\{\ell\in {\mathbb N}: \, 2^\ell > \la/100\}}\beta(2^{-\ell}(\la_i-\tau))}{\la_i^2-\tau^2}
\Bigr)  \ala(\la_i)\, \Bigr|
\lesssim
\begin{cases}
\la^{-2-k} \quad \text{if } \, \, \la_i\le \la
\\
\la^{-2-k}(1+\la_i-\la)^{-N} \quad \text{if } \, \, \la_i>\la.
\end{cases}
$$
Thus for $k=0,1$, by \eqref{b.5}
\begin{align*}
\bigl\| (\la \partial_\tau)^k \, K^+_{\tau,\infty}(\, \cdot \, ,y)\bigr\|_{L^2(M)}
&\lesssim \la^{-2}
\Bigl(\, \sum_{\la_i\le \la} |e^0_i(y)|^2+\sum_{\la_i>\la}(1+\la_i-\la)^{-N}|e^0_i(y)|^2 \,
\Bigr)^{1/2}
\\
&\lesssim \la^{-2+\frac{n}2},
\end{align*}
as desired if $N>2n$, using \eqref{freespec} again.

Next we turn to the bounds in \eqref{3.9} for $K^-_{\tau,\infty}$.  To handle this, let
$\eta$ be as in Lemma~\ref{freeres} and put
\begin{align*}
H^-_{\tau,\infty}(x,y)&=-\sum_i  \Bigl(\frac{\sum_{\{\ell\in {\mathbb N}: \, 2^\ell > \la/100\}}\beta(2^{-\ell}(\la_i-\tau))}{\la_i^2-\tau^2}
\Bigr)
\eta(\la_i/\tau) \, e^0_i(x)e^0_i(y)
\\
&=-\sum_i \frac{\eta(\la_i/\tau)}{\la_i^2-\tau^2} e^0_i(x)e^0_i(y),
\end{align*}
assuming, as we may, that $\la\gg 1$.  The last equality comes from the properties
of our Littlewood-Paley bump function, $\beta$.  We then conclude from Lemma~\ref{freeres} that $H^-_{\tau,\infty}$ satisfies the bounds
in \eqref{3.9}.  If we then set
\begin{multline*}
\tilde K^-_{\tau,\infty}(x,y)=\sum_i \Bigl(\frac{\sum_{\{\ell\in {\mathbb N}: \, 2^\ell > \la/100\}}\beta(2^{-\ell}(\la_i-\tau))}{\la_i^2-\tau^2}
\Bigr) \ala(\la_i) \, e^0_i(x)e^0_i(y)
\\
-\sum_i \Bigl(\frac{\sum_{\{\ell\in {\mathbb N}: \, 2^\ell > \la/100\}}\beta(2^{-\ell}(\la_i-\tau))}{\la_i^2-\tau^2}
\Bigr) \, \bigl(\eta(\la_i/\tau)\bigr) \, e^0_i(x)e^0_i(y),
\end{multline*}
we have $K^-_{\tau,\infty}=\tilde K^-_{\tau,\infty} +H^-_{\tau,\infty}$, and, also, by the proof of \eqref{3.8}, 
$\tilde K^-_{\tau,\infty}$ satisfies the bounds in \eqref{3.9}.

It just remains to prove the bounds in \eqref{3.10} for the $R_{\tau,\ell}(x,y)$ and that in
\eqref{3.11} for $R_{\tau,\infty}(x,y)$.  The former just follows from the proof of \eqref{3.6}.

To prove the remaining inequality, \eqref{3.11}, we note that if $\eta$ is as above and we set
$$\tilde R_{\tau,\infty}(x,y)=\sum_i \frac{\eta(\la_i/\tau)}{\la^2_i-\tau^2} \, e^0_i(x)e^0_i(y),$$
then, by Lemma~\ref{freeres}, $\tilde R_{\tau,\infty}$ satisfies the bounds in \eqref{3.11}.  Also, we have
$$R_{\tau,\infty}(x,y)=R^0_{\tau,\infty}(x,y)+\tilde R_{\tau,\infty}(x,y),$$
if
$$R^0_{\tau,\infty}(x,y)=\sum_i \bigl(1-\eta(\la_i/\tau)\bigr)
 \Bigl(\frac{\sum_{\{\ell\in {\mathbb N}: \, 2^\ell > \la/100\}}\beta(2^{-\ell}(\la_i-\tau))}{\la_i^2-\tau^2}
\Bigr)\, e^0_i(x)e^0_i(y),$$
(again using the properties of $\beta$),
and, since the proof of Lemma~\ref{freeres} shows that for $\tau\in [\la/2,10\la]$ we have
$$|R^0_{\tau,\infty}(x,y)|\lesssim \tau^{n-2} \bigl(1+\tau d_g(x,y)\bigr)^{-N}
\lesssim \la^{n-2} \bigl(1+\la d_g(x,y)\bigr)^{-N},$$
we conclude that \eqref{3.11}
must be valid, which completes the proof.
\end{proof}

\subsection*{Handling the contribution of relatively large frequencies of $H_V$}

In this section we shall handle relatively large 
frequencies of $H_V$ by proving the following.

\begin{proposition}\label{largeprop}
Let $V\in L^1(M)$ with $V^{-}\in
 {\mathcal K}(M)$, and $\ala(\tau)$ be defined as in \eqref{2.15} with $\e$ satisfying \eqref{a.1}. Then we have
\begin{multline}\label{3.21}
\Bigl| \, \sum_{j}
\sum_{\{k: \, \tau_k >10\lambda\}}\int_M \int_M
\frac{\ala(\la_j)-\ala(\tau_k)}{\la_j^2-\tau_k^2} e_j^0(x)e_j^0(y)V(y)e_{\tau_k}(x)
e_{\tau_k}(y) \, dx dy\, \Bigr| 
\\
\le C_V \, 
\la^{n-2}(\log \la)^{1/2},
\end{multline}
for some constant $C_V$ depending on $V$ which is independent of $\e$.
\end{proposition}

To prove \eqref{3.21} fix
\begin{equation}\label{3.22}
\Psi\in C^\infty_0((1/2,2)), \quad \text{with } \, \,
\Psi(s)=1, \, \, s\in [3/4,5/4].
\end{equation}  

To proceed, assume that $\tau_k>10\la$.  Since, by the
mean value theorem and \eqref{2.17}
$$\frac{\ala(\la_j)-\ala(\tau_k)}{\la_j-\tau_k}
=O(\tau_k^{-\sigma}) \, \, \forall \, \sigma,
\, \, \text{if } \, \la_j\in (\tau_k/2,2\tau_k), 
\, \, \tau_k>10\la,$$
by \eqref{freespec} and \eqref{3.13}, to prove \eqref{3.21}
it suffices to show that
\begin{multline}\label{3.21'}\tag{3.21$'$}
\Bigl| \, \sum_{j}
\sum_{\{k: \, \tau_k >10\lambda\}}\iint
\frac{\ala(\la_j)-\ala(\tau_k)}{\la_j^2-\tau_k^2} 
\, \bigl(1-\Psi(\la_j/\tau_k)\bigr)
e_j^0(x)e_j^0(y)V(y)e_{\tau_k}(x)
e_{\tau_k}(y) \, dx dy\, \Bigr| 
\\
\lesssim \|V\|_{L^1(M)} \, 
\la^{n-2}(\log \la)^{1/2},
\end{multline}
since 
\begin{multline*}
\Bigl| \, \sum_{j}
\sum_{\{k: \, \tau_k >10\lambda\}}\iint
\frac{\ala(\la_j)-\ala(\tau_k)}{\la_j^2-\tau_k^2} 
\, \Psi(\la_j/\tau_k)
e_j^0(x)e_j^0(y)V(y)e_{\tau_k}(x)
e_{\tau_k}(y) \, dx dy\, \Bigr| 
\\
\lesssim \la^{-\sigma} \|V\|_{L^1(M)}, \quad \forall \, \sigma.
\end{multline*}

As $\ala(\tau_k)=O(\tau_k^{-\sigma})$ for all $\sigma\in {\mathbb N}$  for $\tau_k>10\la$
and, by Lemma~\ref{freeres},
$$\Bigl|\sum_j \frac{(1-\Psi(\la_j/\tau_k))}{\la_j^2-\tau_k^2}
e^0_j(x)e^0_j(y)\Bigr|
\lesssim 
\begin{cases}
\tau_k^{n-2}+(d_g(x,y))^{2-n}, \quad n\ge3
\\
\log\bigl(2+1/(\tau_k d_g(x,y))\bigr), \quad n=2,
\end{cases}
$$
the analog of \eqref{3.21'} where we replace
$(\ala(\la_j)-\ala(\tau_k))$ by $\ala(\tau_k)$ is trivial.
Consequently, we would have \eqref{3.21'} and consequently
Proposition~\ref{largeprop} if we could show that
\begin{multline}\label{3.21''}\tag{3.21$''$}
\Bigl| \, \sum_{j}
\sum_{\{k: \, \tau_k >10\lambda\}}\iint
\frac{(1-\Psi(\la_j/\tau_k))}{\la_j^2-\tau_k^2} 
\, \ala(\la_j) \, 
e_j^0(x)e_j^0(y)V(y)e_{\tau_k}(x)
e_{\tau_k}(y) \, dx dy\, \Bigr| 
\\
\lesssim \|V\|_{L^1(M)} \, 
\la^{n-2}(\log \la)^{1/2}.
\end{multline}

If $1-\Psi(\la_j/\tau_k)\ne 0$ we have $\la_j\ne \tau_k$, and then can write
$$
\frac{1}{\tau_k^2-\la_j^2}=\tau_k^{-2}
+\tau_k^{-2}\bigl(\la_j/\tau_k\bigr)^2+
\cdots + \tau_k^{-2}\bigl(\la_j/\tau_k\bigr)^{2N-2}
\\
+(\la_j/\tau_k)^{2N}\frac{1}{\tau_k^2-\la_j^2}.
$$
As a result, we would have \eqref{3.21''} if we could
choose $N\in {\mathbb N}$ so that we have
\begin{multline}\label{3.23}
\Bigl| \iint \sum_j
\la_j^{2\ell}
\ala(\la_j)e_j^0(x)e_j^0(y)
\Bigl(\sum_{\tau_k>10\la}\tau_k^{-2-2\ell}
(1-\Psi(\la_j/\tau_k)) \, e_{\tau_k}(x)e_{\tau_k}(y)\Bigr) \, 
V(y) \, dxdy\Bigr|
\\
\lesssim \|V\|_{L^1(M)}
\la^{n-2} (\log\la)^{1/2}, \quad \ell=0,\dots,N-1,
\end{multline}
as well as
\begin{multline}\label{3.24}
\Bigl|\sum_j\sum_{\tau_k>10\la}\iint
\frac{(1-\Psi(\la_j/\tau_k))}{\la_j^2-\tau_k^2}
(\la_j)^{2N}\ala(\la_j) e^0_j(x)e^0_j(y)V(y)\tau_k^{-2N}
e_{\tau_k}(x)e_{\tau_k}(y) \, dxdy\Bigr|
\\
\lesssim \la^{n-2}  \|V\|_{L^1(M)}.
\end{multline}

To handle \eqref{3.23} we start with a trivial reduction.
We note that if $\tau_k>10\la$, then by 
\eqref{2.16}, \eqref{3.13} and \eqref{3.22}
\begin{align*}
\bigl|\ala(\la_j)\sum_{\tau_k>10\la}\Psi(\la_j/\tau)
\tau_k^{-2-2\ell}e_{\tau_k}(x)e_{\tau_k}(y)|&\lesssim
\bigl|\ala(\la_j)\sum_{\tau_k\in (\la_j/2,2\la_j)}
\tau_k^{-2-2\ell}e_{\tau_k}(x)e_{\tau_k}(y)\bigr|
\\
&\lesssim \la_j^{-\sigma}\sum_{\tau_k\approx \la_j}
\tau_k^{-\sigma}|\tau_k^{-2-2\ell}e_{\tau_k}(x)e_{\tau_k}(y)|
\\
&\lesssim \la_j^{n-2\ell-2\sigma},
\end{align*}
for any $\sigma$.
If $\sigma>n$, by \eqref{freespec} this yields
\begin{multline*}
\Bigl| \iint \sum_j
\la_j^{2\ell}
\ala(\la_j)e_j^0(x)e_j^0(y)
\Bigl(\sum_{\tau_k>10\la}\tau_k^{-2-2\ell}
\Psi(\la_j/\tau_k) \, e_{\tau_k}(x)e_{\tau_k}(y)\Bigr) \, 
V(y) \, dxdy\Bigr|
\\
\lesssim \|V\|_{L^1(M)},
\end{multline*}
which means that in order to prove \eqref{3.23} it
suffices to show that
\begin{multline}\label{3.23'}\tag{3.23$'$}
\Bigl| \iint \bigl((P^0)^{2\ell}\ala(P^0)\bigr)(x,y)
\Bigl(\sum_{\tau_k>10\la}\tau_k^{-2-2\ell}
\, e_{\tau_k}(x)e_{\tau_k}(y)\Bigr) \, 
V(y) \, dxdy\Bigr|
\\
\lesssim \|V\|_{L^1(M)}
\la^{n-2} (\log\la)^{1/2}, \quad \ell=0,\dots,N-1,
\end{multline}
since
$$
\sum_j \la_j^{2\ell}\ala(\la_j)e_j^0(x)e_j^0(y)=
\bigl((P^0)^{2\ell}\ala(P^0)\bigr)(x,y).
$$

To prove this, in certain cases, we shall rewrite
the expression inside the absolute value in the left side of \eqref{3.23'}
slightly.
Specifically, we can split it into the following two terms
\begin{multline}\label{3.25}
\iint \bigl((P^0)^{2\ell}\ala(P^0)\bigr)(x,y)
\Bigl(\sum_{\tau_k\ge1}\tau_k^{-2-2\ell}
\, e_{\tau_k}(x)e_{\tau_k}(y)\Bigr) \, 
V(y) \, dxdy
\\
-\iint \bigl((P^0)^{2\ell}\ala(P^0)\bigr)(x,y)
\Bigl(\sum_{\tau_k\le 10\la}\tau_k^{-2-2\ell}
\, e_{\tau_k}(x)e_{\tau_k}(y)\Bigr) \, 
V(y) \, dxdy
\\
=I+II, \quad \text{if } \, \ell \le (n-4)/4 \, \, 
\text{and } \, n\ge4.
\end{multline}
If $n\le 3$ we shall not split things up in this way, and,
instead, just deal with the expression in the left
side of \eqref{3.23'} directly.

Note that if $n\ge5$ and $\ell\le (n-4)/4$
\begin{align*}
|I|&=\Bigl|\iint \bigl((P^0)^{2\ell}\ala(P^0)\bigr)(x,y)
\, \bigl(H_V\bigr)^{-1-\ell}(x,y) \, V(y)\, dxdy\Bigr|
\\
&\le\iint_{d_g(x,y)\le \la^{-1}}
+\iint_{d_g(x,y)\ge \la^{-1}}
\Bigl( \, 
\bigl|\bigl((P^0)^{2\ell}\ala(P^0)\bigr)(x,y)\bigr|\,
\, \bigl|\bigl(H_V\bigr)^{-1-\ell}(x,y)\bigr| \, |V(y)|\, dxdy
\\
&\lesssim \iint_{d_g(x,y)\le \la^{-1}}
\la^{n+2\ell} \, (d_g(x,y))^{-n+2+2\ell} \, |V(y)| \, dxdy
\\
&\qquad+\|V\|_{L^1}\cdot \sup_y
\bigl(\int_{d_g(x,y)\ge \la^{-1}} \bigl|\bigl((P^0)^{2\ell}\ala(P^0)\bigr)(x,y)\bigr|^2 \, dx\bigr)^{1/2}
\\
&\qquad\qquad\qquad \qquad\qquad\times
\bigl(\int_{d_g(x,y)\ge \la^{-1}} \bigl| \bigl(
d_g(x,y)\bigr)^{-2(n-2-2\ell)} \, dx\bigr)^{1/2}
\\
&\lesssim \|V\|_{L^1}\cdot \la^{n+2\ell-(2-2\ell)}
+\|V\|_{L^1}\cdot \bigl(\la^{\frac{n}2+2\ell}
\cdot \la^{\frac{n}2-2-2\ell}\bigr)
\\
&=\la^{n-2}\|V\|_{L^1},
\end{align*}
which is better than the bounds in \eqref{3.23'}.
Here we used Lemma~\ref{heat} to bound $(H^{-1-\ell}_V)(x,y)$ (and our momentary assumption $\ell\le (n-4)/4$).
In the second inequality we also used Schwarz's inequality, while in the second inequality and the second to last step we also used Lemma~\ref{ala}.

If $n=4$ than the requirement in \eqref{3.25} forces
$\ell=0$.  In this case, if we repeat the above arguments
we obtain slightly worse bounds, i.e.,
$$|I|\lesssim \la^{n-2}(\log\la)^{1/2} \|V\|_{L^1},$$
with the $\log\la$ factor coming from the fact that when
$n=4$ we have
$$\int_{d_g(x,y)\ge \la^{-1}}(d_g(x,y))^{-4}\, dx
\approx \log \la.$$
On the other hand, this bound is in agreement with  the one posited in \eqref{3.23'}.

We still need to handle the second term, $II$, in \eqref{3.25}.  To do this we shall again
use Lemma~\ref{ala} and \eqref{3.13} along with Schwarz's inequality to deduce that
\begin{align*}
| II |  &\le \|V\|_{L^1}\cdot \sup_y \Bigl( \, 
\|(P^0)^{2\ell} \ala(P^0)(\, \cdot, \, y)\|_{L^2}\cdot 
\|\sum_{\tau_k\le 10\la} \tau^{-2-2\ell}_k e_{\tau_k}(\, \cdot\, )e_{\tau_k}(y)\|_{L^2}\Bigr)
\\
&\lesssim \|V\|_{L^1} \cdot \la^{\frac{n}2+2\ell} \cdot \bigl(\sum_{\tau_k\le 10\la}
\tau_k^{-4-4\ell}|e_{\tau_k}(y)|^2\bigr)^{1/2}
\\
&\lesssim   \|V\|_{L^1} \cdot \la^{\frac{n}2+2\ell} \cdot 
\bigl(\sum_{ \{j\in {\mathbb N}: \, 2^j\le 10\la\} } 2^{-j(4+4\ell)} 2^{nj}\bigr)^{1/2}
\\
&\lesssim \|V\|_{L^1} \cdot \la^{\frac{n}2+2\ell} \cdot \la^{-2-2\ell+\frac{n}2} 
\\
&=\|V\|_{L^1} \cdot \la^{n-2},
\end{align*}
assuming in the last step $\ell<(n-4)/4$.  In the remaining case covered in \eqref{3.25}
where $n\ge4$ and $\ell=(n-4)/4$ (forcing $n$ to be a multiple of $4$), as was the case 
for $\ell=0$ and $n=4$, the bound is somewhat worse and we instead get, in this case,
$$| II |\lesssim \la^{n-2}(\log\la)^{1/2} \|V\|_{L^1},$$
which still better than that of our current goal, \eqref{3.21}.

Since we have obtained favorable estimates for $I$ and $II$ in \eqref{3.25}, we have shown that \eqref{3.23'} is valid
when $n\ge 4$ and $\ell \le (n-4)/4$.  For the remaining cases where $n=2, 3$ and $0\le \ell \le N-1$
is arbitrary or $(n-4)/4<\ell \le N-1$ for $n\ge4$, we shall just repeat the argument that we used
to control $II$.  We have not specified $N$; however, to get the other inequality, \eqref{3.24}, that is needed to obtain our current goal \eqref{3.21}, $N$ will have to be chosen to be larger than $(n-4)/4$.

In these remaining cases for \eqref{3.23'} if we argue as above we find that the left side of \eqref{3.23'} is dominated by
\begin{align*}
\|V\|_{L^1}\cdot &\sup_y \Bigl(\|(P^0)^{2\ell} \ala(P^0)(\, \cdot, \, y)\|_{L^2}\cdot 
\|\sum_{\tau_k> 10\la} \tau^{-2-2\ell}_k e_{\tau_k}(\, \cdot\, )e_{\tau_k}(y)\|_{L^2}\Bigr)
\\
&\lesssim \|V\|_{L^1}\cdot \la^{\frac{n}2+2\ell} \cdot \bigl(\sum_{\tau_k> 10\la}
\tau_k^{-4-4\ell}|e_{\tau_k}(y)|^2\bigr)^{1/2}
\\
&\lesssim   \|V\|_{L^1} \cdot \la^{\frac{n}2+2\ell} \cdot
\bigl(\sum_{ \{j\in {\mathbb N}: \, 2^j> 10\la\} } 2^{-j(4+4\ell)} 2^{nj}\bigr)^{1/2}
\\
&\lesssim \|V\|_{L^1} \cdot \la^{\frac{n}2+2\ell} \cdot \la^{-2-2\ell+\frac{n}2} 
\\
&= \la^{n-2}\|V\|_{L^1},
\end{align*}
using the fact that our current conditions ensure that $4+4\ell >n$.  This completes the proof
of \eqref{3.23'} and hence \eqref{3.23}.

To finish this subsection we need to show that we can fix $N\in {\mathbb N}$ sufficiently large so that
\eqref{3.24} is valid.  As we mentioned before we shall specify our $N>(n-4)/4$ in a moment.

To prove \eqref{3.24} for large enough $N$, here too it will be convenient to split matters into two cases.  First,
let us deal with the sum in \eqref{3.24} where $\tau_k\ge \la^2$.  We can handle this case using trivial methods if
$N$ is large enough.  In fact, by using 
Schwarz's inequality,
Lemma~\ref{ala} and orthogonality, we see that for $\tau_k\ge \la^2$ we
have the uniform bounds
\begin{align*}
\int\bigl| \, \sum_j   
&\frac{(1-\Psi(\la_j/\tau_k))}{\la_j^2-\tau_k^2}
(\la_j)^{2N}\ala(\la_j) e^0_j(x)e^0_j(y) \, \bigr| \, dx
\\
&\lesssim \bigl\| \sum_j
\frac{(1-\Psi(\la_j/\tau_k))}{\la_j^2-\tau_k^2}
(\la_j)^{2N}\ala(\la_j) e^0_j(\, \cdot \, )e^0_j(y)\bigr\|_{L^2}
\\
&\lesssim \bigl\| (P^0)^{2N} \ala(P^0)(\, \cdot ,y)\bigr\|_{L^2} \lesssim \la^{\frac{n}2+2N}.
\end{align*}
This does not present a problem since if $N$ is large, since, by \eqref{3.13} and the Cauchy-Schwartz inequality,
$$\sum_{\tau_k\ge \la^2} \tau_k^{-2N}|e_{\tau_k}(x)e_{\tau_k}(y)|\lesssim
\sum_{ \{j \in {\mathbb N}: \, 2^j\ge \la^2\} } 2^{-2Nj}2^{nj}
\lesssim \la^{-4N}\la^{2n},$$
if $N>n$.  Using these two inequalities we deduce that
\begin{multline*}
\Bigl|\sum_j\sum_{\tau_k\ge \la^2}\iint
\frac{(1-\Psi(\la_j/\tau_k))}{\la_j^2-\tau_k^2}
(\la_j)^{2N}\ala(\la_j) e^0_j(x)e^0_j(y)V(y)\tau_k^{-2N}
e_{\tau_k}(x)e_{\tau_k}(y) \, dxdy\Bigr|
\\
\lesssim \la^{\frac{n}2+2N} \la^{-4N+2n}  \|V\|_{L^1(M)} < \|V\|_{L^1},
\end{multline*}
if we assume, as we may, that $N=2n$.

Based on this, we would be done with handling relatively large frequencies of $H_V$ if we could show that
\begin{multline}\label{3.24'}\tag{3.24$'$}
\Bigl|\sum_j\sum_{10\la<\tau_k<\la^2}\iint
\frac{(1-\Psi(\la_j/\tau_k))}{\la_j^2-\tau_k^2}
(\la_j)^{2N}\ala(\la_j) e^0_j(x)e^0_j(y)V(y)\tau_k^{-2N}
e_{\tau_k}(x)e_{\tau_k}(y) \, dxdy\Bigr|
\\
\lesssim \la^{n-2}  \|V\|_{L^1(M)}, \quad \text{if } \, \, N=2n.
\end{multline}

To this end, let $\Phi\in C^\infty(\R_+)$ satisfy $\Phi(s)=1$, $s\le 3/2$ and $\Phi(s)=0$, $s\ge2$.  Then
if $\tau_k\ge 10\la$, it follows that
$$\Phi(\la_j/\la)\bigl(1-\Psi(\la_j/\tau_k)\bigr)=\Phi(\la_j/\la),$$
and also for all $\sigma\in {\mathbb N}$
$$\ala(\la_j)\frac{(1-\Phi(\la_j/\la))}{\la^2_j-\tau_k^2}\bigl(1-\Psi(\la_j/\tau_k)\bigr)=O(\tau_k^{-\sigma}),
\, \, \, \text{if } \, \, 10\la\le \tau_k\le \la^2.$$
Thus, by an earlier argument, modulo $O(\la^{-\sigma}\|V\|_{L^1})$ $\forall \, \sigma\in {\mathbb N}$, the left
side of \eqref{3.24'} agrees with the expression where we replace $(1-\Psi(\la_j/\tau_k))$ with
$\Phi(\la_j/\la)$.  Therefore since $(\la_j^2-\tau_k^2)^{-1}=-(1-(\la_j/\tau_k)^2)^{-1} \cdot \tau^{-2}_k$, we
would have \eqref{3.24'} if we could show that
\begin{multline}\label{3.24''}\tag{3.24$''$}
\Bigl|\sum_j\sum_{10\la<\tau_k<\la^2}\iint
\frac{\Phi(\la_j/\la)}{1-(\la_j^2/\tau_k)^2}
(\la_j)^{2N}\ala(\la_j) e^0_j(x)e^0_j(y)V(y)\tau_k^{-2N-2}
e_{\tau_k}(x)e_{\tau_k}(y) \, dxdy\Bigr|
\\
\lesssim \la^{n-2}  \|V\|_{L^1(M)}, \quad N=2n.
\end{multline}

Since the left side is dominated by $\|V\|_{L^1}$ times
\begin{multline}\label{3.26}
\sup_y\Bigl|  \int \Bigl(\, \sum_j \frac{\Phi(\la_j/\la)}{1-\tau_k^{-2}\la_j^2}
(\la_j)^{2N}\ala(\la_j) e^0_j(x)e^0_j(y)\, \Bigr) 
\\
\times
\bigl(\, 
\sum_{10\la<\tau_k<\la^2}
\tau_k^{-2N-2}
e_{\tau_k}(x)e_{\tau_k}(y)\, \bigr) \, dx \Bigr|,
\end{multline}
it suffices to show that this expression is 
$O(\la^{n-2})$.

To do so, we shall appeal to the $\delta_\tau$--Lemma, Lemma~\ref{delta}.  We set for a given $y\in M$
$$m(s,x)=\sum_j  \frac{\Phi(\la_j/\la)}{1-s^2\la_j^2}
(\la_j)^{2N}\ala(\la_j) e^0_j(x)e^0_j(y), \quad s\in [0,1/10\la].$$
Then, since
\begin{equation}\label{3.27}
\Bigl|  \frac{\Phi(\la_j/\la)}{1-s^2\la_j^2} \, \Bigr| \lesssim 1,
\end{equation}
and
\begin{equation}\label{3.28}
\Bigl|\frac\partial{\partial s}  \Bigl( \, \frac{\Phi(\la_j/\la)}{1-s^2\la_j^2}\, \Bigr) \, \Bigr| \lesssim s\la^2,
\end{equation}
if $s\in [0,1/10\la]$, it follows from orthogonality and Lemma~\ref{ala} that
\begin{equation}\label{3.29}
\|m(0,\, \cdot\, )\|_{L^2(M)}
+\int_0^{1/10\la}\bigl\| \tfrac\partial{\partial s}m(s,\, \cdot \, )\bigr\|_{L^2(M)} \, ds
=O(\la^{\frac{n}2+2N}).
\end{equation}
Consequently, by Lemma~\ref{delta}, \eqref{3.26} is dominated by
\begin{align*}
\la^{\frac{n}2+2N}\bigl\|\sum_{10\la <\tau_k<\la^2}
\tau_k^{-2N-2} e_{\tau_k}(\, \cdot \, )e_{\tau_k}(y)\bigr\|_{L^2}
&=\la^{\frac{n}2+2N}  \bigl(\, |\sum_{10\la <\tau_k<\la^2}
\tau_k^{-4N-4}|e_{\tau_k}(y)|^2 \bigr)^{1/2}
\\
&\lesssim \la^{\frac{n}2+2N}
\bigl( \sum_{ \{j\in {\mathbb N}: \, 2^j>10\la\} } 2^{-(4N+4)j} 2^{nj} \bigr)^{1/2}
\\
&\lesssim \la^{\frac{n}2+2N}\cdot \la^{-2N-2}\cdot \la^{\frac{n}2}=\la^{n-2},
\end{align*}
using \eqref{3.13} in the second to last step and the fact that $N=2n$ in the final one.
Thus, the quantity in \eqref{3.26} is $O(\la^{n-2})$,
which, by the above, yields
\eqref{3.24''} and finishes the proof of Proposition~\ref{largeprop}.

\subsection*{Handling the contribution of relatively small frequencies of $H_V$}

In this subsection we shall handle relatively small frequencies of $H_V$ and prove the following result.

\begin{proposition}\label{smallprop}
Let $V\in L^1(M)$ with $V^{-}\in
 {\mathcal K}(M)$, and $\ala(\tau)$ be defined as in \eqref{2.15} with $\e$ satisfying \eqref{a.1}. Then we have
\begin{multline}\label{3.30}
\Bigl| \, \sum_{j}
\sum_{ \tau_k<\lambda/2}\int_M \int_M
\frac{\ala(\la_j)-\ala(\tau_k)}{\la_j^2-\tau_k^2} e_j^0(x)e_j^0(y)V(y)e_{\tau_k}(x)
e_{\tau_k}(y) \, dx dy\, \Bigr| 
\\
\le C_V \la^{n-2}
,
\end{multline}
for some constant $C_V$ depending on $V$ which is independent of $\e$.
\end{proposition}

If we combine this with Proposition~\ref{compprop} and
Proposition~\ref{largeprop} which handle frequencies which are comparable to $\la$ and large compared to $\la$, respectively, we obtain Proposition~\ref{mainprop}, which by the arguments in \S 2, yield our main result, 
Proposition~\ref{abstract}.

\begin{proof}[Proof of Proposition~\ref{smallprop}]
As in the earlier cases, we shall first handle a trivial case.  To do so, we note that, by \eqref{2.17} and the mean value theorem,
$$\frac{\ala(\la_j)-\ala(\tau)}{\la_j^2-\tau^2}
=O(\la^{-\sigma}), \, \, \forall \, \sigma\in 
{\mathbb N}, \quad \text{if } \, \, 
1\le \tau\le \la/2 \, \, \text{and} \, \,
\la_j\le 7\la/8.
$$
Also, by \eqref{freespec} and \eqref{3.13}
\begin{equation}\label{3.31}\sum_{\la_j\le \la}|e^0_j(x)e^0_j(y)|, \, \,
\sum_{\tau_k< \la/2} |e_{\tau_k}(x) e_{\tau_k}(y)|
\lesssim \la^n.
\end{equation}
To use these and make our first reduction 
fix $a\in C^\infty(\R_+)$ satisfying
$$a(s)=0, \quad s\le 3/4 \quad \text{and } \, \,
a(s)=1, \, \, s\ge 7/8.$$
Using the preceding inequalities we see that in order to prove
\eqref{3.30} it suffices to show that
\begin{multline}\label{3.30'}\tag{3.30$'$}
\Bigl| \, \sum_{j}
\sum_{ \tau_k<\lambda/2}\int_M \int_M
\frac{\ala(\la_j)-\ala(\tau_k)}{\la_j^2-\tau_k^2}
a(\la_j/\la) e_j^0(x)e_j^0(y)V(y)e_{\tau_k}(x)
e_{\tau_k}(y) \, dx dy\, \Bigr| 
\\
\le C_V \la^{n-2} \|V\|_{L^1(M)}
,
\end{multline}
due to the fact that the difference between the quantities inside the absolute values in the left side of \eqref{3.30} and that of \eqref{3.30'} is 
$O(\la^{-\sigma}\|V\|_{L^1})$ for all $\sigma$.

For the next reduction, note that the proof of Lemma~\ref{freeres} yields
$$\Bigl|\, 
\sum_j
\frac{a(\la_j/\la)}{\la_j^2-\tau_k^2}
e_j^0(x)e_j^0(y)\, \Bigr|
\lesssim 
\begin{cases}(d_g(x,y))^{2-n}, \quad n\ge 3,
\\
\log\bigl(2+(d_g(x,y))^{-1}\bigr), \quad n=2,
\end{cases}
$$
if $1\le\tau_k\le \la/2$.
Based on this and the second part of \eqref{3.31}
and the fact that $1-\ala(\tau_k)=O(\la^{-\sigma})$ for
all $\sigma$ when $1\le \tau_k\le \la/2$, we easily see that
\begin{multline*}
\Bigl| \, \sum_{j}
\sum_{ \tau_k<\lambda/2}\int_M \int_M
\frac{1-\ala(\tau_k)}{\la_j^2-\tau_k^2}
a(\la_j/\la) e_j^0(x)e_j^0(y)V(y)e_{\tau_k}(x)
e_{\tau_k}(y) \, dx dy\, \Bigr| 
\\
\le \la^{-\sigma} \|V\|_{L^1(M)}, \quad \forall \, \sigma
\in {\mathbb N}.
\end{multline*}
Consequently, we would have \eqref{3.30'} if we could show
that
\begin{multline}\label{3.30''}\tag{3.30$''$}
\Bigl| \, \sum_{j}
\sum_{ \tau_k<\lambda/2}\int_M \int_M
\frac{a(\la_j/\la)}{\la_j^2-\tau_k^2} \, \bigl(\ala(\la_j)-1\bigr)\, 
e_j^0(x)e_j^0(y)V(y)e_{\tau_k}(x)
e_{\tau_k}(y) \, dx dy\, \Bigr| 
\\
\le C_V \la^{n-2} \|V\|_{L^1(M)}.
\end{multline}

We need to make one final reduction before we can
appeal to the $\delta_\tau$--Lemma, Lemma~\ref{delta}.
For this, let $\eta$ be as in Lemma~\ref{freeres}, i.e.,
$\eta\in C^\infty(R_+)$ with $\eta(s)=0$ on
$s\le 2$ and $\eta(s)=1$, $s>4$.  It then follows
that
$$\eta(s)a(s)=\eta(s).$$
Consequently, we can write the quanity inside the
absolute value in the left side of \eqref{3.30''} as
\begin{multline*}
\sum_{j}
\sum_{ \tau_k<\lambda/2}\int_M \int_M
\frac{\eta(\la_j/\la)}{\la_j^2-\tau_k^2} \, \bigl(\ala(\la_j)-1\bigr)\, 
e_j^0(x)e_j^0(y)V(y)e_{\tau_k}(x)
e_{\tau_k}(y) \, dx dy
\\
+ \sum_{j}
\sum_{ \tau_k<\lambda/2}\int_M \int_M
\frac{a(\la_j/\la)(1-\eta(\la_j/\la))}{\la_j^2-\tau_k^2} \, \bigl(\ala(\la_j)-1\bigr)\, 
e_j^0(x)e_j^0(y)V(y)e_{\tau_k}(x)
e_{\tau_k}(y) \, dx dy 
\\
= I + II.
\end{multline*}
Therefore, in order to prove \eqref{3.30''}, it suffices to show that both $|I|$ and $|II|$ are dominated by the
right side of \eqref{3.30''}.

We can easily handle $I$ without appealing to the
$\delta_\tau$--lemma.
Indeed since $\ala(\la_j)=O(\la_j^{-\sigma})$ for all
$\sigma$ if $\eta(\la_j/\la)\ne0$, we see that
\eqref{freespec} yields
$$\sum_j \frac{\eta(\la_j/\la)}{\la_j^2-\tau_k^2} \, \bigl(\ala(\la_j)-1\bigr)\, 
e_j^0(x)e_j^0(y) =
-\sum_j \frac{\eta(\la_j/\la)}{\la_j^2-\tau_k^2} \, 
e_j^0(x)e_j^0(y) +O(\la^{-\sigma}), \quad \forall \,
\sigma.$$
Consequently, by the second part of \eqref{3.31} we would have the desired bounds for $I$ if we could show that
\begin{equation}\label{3.32}\Bigl|\iint \sum_{\tau_k< \la/2} R_{\tau_k}(x,y)
e_{\tau_k}(x)e_{\tau_k}(y) \, V(y)\, dxdy\Bigr|
\lesssim \la^{n-2} \|V\|_{L^1},
\end{equation}
where
$$R_{\tau_k}(x,y)= \sum_j \frac{\eta(\la_j/\la)}{\la_j^2-\tau_k^2} \, 
e_j^0(x)e_j^0(y).$$
To use this we note that the proof of Lemma~\ref{freeres}
implies that
$$\sup_{1\le \tau_k< \la/2}|R_{\tau_k}(x,y)|
\le C_0 \la^{n-2}h_n(\la d_g(x,y))(1+\la d_g(x,y))^{-\sigma}, 
\quad \forall \, \sigma,$$
and, therefore,
$$\sup_y \int \sup_{1\le \tau_k< \la/2}|R_{\tau_k}(x,y)| \, dx \lesssim \la^{-2}.$$
Since, we always have $\tau_k\ge 1$ by \eqref{1.7},
by the second part of \eqref{3.31} we have 
$$\sup_y\int\Bigl| \sum_{\tau_k<\la/2} R_{\tau_k}(x,y)e_{\tau_k}(x)
e_{\tau_k}(y)\Bigr|\, dx
\lesssim \la^n \cdot \sup_y\int \sup_{1\le \tau_k< \la/2}|R_{\tau_k}(x,y)|\, dx \lesssim \la^{n-2},$$
which clearly yields \eqref{3.32}.

Since we have the desired estimate for $I$ above, it only remains to prove the corresponding estimate for $II$.  
For this, let 
\begin{multline*}
m(s,x,y)=\sum_j \frac{b(\la_j/\la)}{\la_j^2-s^2} \, \bigl(\ala(\la_j)-1\bigr)\, 
e_j^0(x)e_j^0(y),
\\\text{with } \, b(s)=
a(s)(1-\eta(s))\in C^\infty((3/4,4)) .
\end{multline*}
We then can rewrite this desired bound for $II$ as follows
\begin{equation*}
\Bigl| \iint 
\sum_{\tau_k<\la/2} m(\tau_k,x,y) e_{\tau_k}(x) e_{\tau_k}(y) \, V(y)\, dx dy\Bigr|
\lesssim \la^{n-2}\|V\|_{L^1}.
\end{equation*}
Just as the last step in the proof of Proposition~\ref{largeprop} was to establish \eqref{3.26}, the final step here would be to show that
\begin{equation} \label{3.33}
\sup_y  \int \bigl|\sum_{\tau_k<\la/2} m(\tau_k,x,y) e_{\tau_k}(x) e_{\tau_k}(y)\bigr| \, dx
\lesssim \la^{n-2}. 
\end{equation}

To prove this we shall argue as in the very end of the
last subsection and appeal to the $\delta_\tau$--Lemma~\ref{delta} with the $\delta$ there equal
to $\la/2$.
We first note that by \eqref{freespec} 
and the fact that $b(\la_j/\la)\ne 0$ implies $\la_j/ \la\in [3/4,4]$.  Consequently,
$$\Bigl| \, \Bigl(\frac{\partial}{\partial s}\Bigr)^\ell \, \Bigl(
 \frac{b(\la_j/\la)}{\la_j^2-s^2} \Bigr) \Bigr| \le C
\la^{-2}(s\la^{-2})^\ell, \, \, \ell =0,1,
\quad \text{if } \, 1\le s\le \la/2.$$
Using this and the support properties $b$ we can
easily see that by the proof of Lemma~\ref{ala} 
that \eqref{freespec} and orthogonality yields
for $\ell=0,1$
$$\big\|(\tfrac\partial{\partial s})^\ell m(s,\, \cdot \, ,y)\bigr\|_{L^2(M)}\le C_0
\la^{\frac{n}2-2} (s\la^{-2})^\ell,  
\quad \text{if } \, \, y\in M, \, \, 0\le s\le \la/2.$$

Consequently,
$$\sup_y \Bigl(\|m(1,\, \cdot \, ,y)\|_{L^2(M)}
+\int_1^{\la/2} \big\|(\tfrac\partial{\partial s}) m(s,\, \cdot \, ,y)\bigr\|_{L^2(M)}\Bigr)
\lesssim \la^{\frac{n}2-2}.$$
By Lemma~\ref{delta} and the second part of \eqref{3.31} we deduce from this
that the left side of \eqref{3.33} is dominated by
$$\la^{\frac{n}2-2} \sup_y \bigl(\sum_{\tau_k<\la/2}
|e_{\tau_k}(y)|^2 \, \bigr)^{1/2}\lesssim \la^{n-2},$$
which completes the proof. 
\end{proof}

\subsection*{Proof of Lemma~\ref{smallinterval}}
\quad\newline

To prove \eqref{2.4},  let us fix a non-negative function $\chi\in \mathcal{S}(\mathbb{R})$ satisfying:
\begin{equation}\label{4.2}
 \chi(\tau)>1, \hspace{1mm} |\tau|\leq 1 \hspace{2mm} \text{and} \hspace{2mm} \hat{\chi}(t)=0, \hspace{1mm} |t|\geq 1/2.
\end{equation}
Then it suffices to show that
\begin{equation}\label{4.3}
\int_M \sum\limits_{k=1}\limits^{\infty} \chi(\e^{-1}(\la-\tau_k))|e_{\tau_k}(x)|^2 dx \le C_V(\e\la^{n-1}+\e^{-1}\la^{n-\frac32})
\end{equation}

By Euler's formula we can rewrite the left side of \eqref{4.3} as 
\begin{equation}\label{4.3'}
\frac1\pi\int_{-\infty}^{\infty}\int_M \e\hat\chi(\e t)e^{it\la}\sum\limits_{k=1}\limits^{\infty}\cos t\tau_k |e_{\tau_k}(x)|^2 dx dt
\end{equation}
minus
$$ \int_M \sum\limits_{k=1}\limits^{\infty} \chi(\e^{-1}(\la+\tau_k))|e_{\tau_k}(x)|^2 dx.
$$
Since $\chi\in \mathcal{S}(\mathbb{R})$, the last term is rapidly decreasing in $\la$ with bounds independent of $\e\le 1$. Thus
we just need to show that the expression in \eqref{4.3'} is bounded by the right side of \eqref{4.3}.

On the other hand, by \eqref{a.3}, it is straightforward to see that 
\begin{equation}\label{a.5}
\int_M \sum\limits_{j=1}\limits^{\infty} \chi(\e^{-1}(\la-\la_j))|e^0_{j}(x)|^2 dx \le C\e\la^{n-1},
\end{equation}
as well as 
\begin{equation}\label{a.6}
 \int_M \sum\limits_{j=1}\limits^{\infty} \chi(\e^{-1}(\la+\la_j))|e^0_{j}(x)|^2 dx \le C_N \la^{-N}, \,\,\forall\,\,N.
\end{equation}
By using Euler's formula again, we have
\begin{equation}\label{4.4}
\frac1\pi\int_{-\infty}^{\infty}\int_M \e\hat\chi(\e t)e^{it\la}\sum\limits_{j=1}\limits^{\infty}\cos t\la_j |e^0_{j}(x)|^2 dx dt  \le C\e\la^{n-1}.
\end{equation}


By repeating the previous arguments using Lemma \ref{triglemma} and Duhamel's principle,  we can rewrite the difference of \eqref{4.3'} and \eqref{4.4} as
\begin{equation}\label{4.5}
 \, \sum_{j,k}\int_M \int_M
\frac{\tilde \chi_\la(\la_j)-\tilde \chi_\la(\tau_k)}{\la_j^2-\tau_k^2} e_j^0(x)e_j^0(y)V(y)e_{\tau_k}(x)
e_{\tau_k}(y) \, dx dy,
\end{equation}
where 
\begin{equation}\label{4.5'}
\tilde \chi_\la(\tau)=\frac1\pi\int_{-\infty}^{\infty} \e\hat\chi(\e t)e^{it\la}\cos \tau \hspace{1mm}  dt= \chi(\e^{-1}(\la-\tau))+ \chi(\e^{-1}(\la+\tau))
\end{equation}
and similarly, we interpret
\begin{equation}\label{4.6}
\frac{\tilde \chi_\la(\tau)-\tilde \chi_\la(\mu)}{\tau^2-\mu^2}=\frac{\tilde \chi_\la^\prime(\tau)}{2\tau},  \hspace{2mm} \text{if}\,\, \,\,\tau=\mu
\end{equation}
Since $\chi\in \mathcal{S}(\mathbb{R})$, we have 
\begin{equation}\label{4.7}
\bigl(\tfrac{d}{d\tau}\bigr)^j \,  \tilde \chi_\la(\tau)
= O\bigl(\e^{-j}(1+\e^{-1}|\la-\tau|)^{-N}\bigr)
\, \, \, \forall \, N, \quad
\text{if } \, j=0,1,2,3,\dots.
\end{equation}

Thus, the proof of Lemma~\ref{smallinterval} would be complete if we could prove the following
\begin{multline}\label{4.8}
\Bigl| \, \sum_{j,k}\int_M \int_M
\frac{\tilde \chi_\la(\la_j)-\tilde \chi_\la(\tau_k)}{\la_j^2-\tau_k^2} e_j^0(x)e_j^0(y)V(y)e_{\tau_k}(x)
e_{\tau_k}(y) \, dx dy \Bigr|
\le C_{V}\e^{-1} \, 
\la^{n-\frac32},
\end{multline}

The proof of \eqref{4.8} is completely analogous to the proof of Proposition~\ref{mainprop}, since the properties of the function $\ala(\tau)$ we used in the proof of \eqref{2.24} are all satisfied by 
the function $\tilde \chi_\la(\tau)$ as well, e.g., \eqref{2.16}, \eqref{2.17} and Lemma~\ref{ala}. We skip the proof of \eqref{4.8} here for the sake of brevity.

\newsection{Improved reminder term estimates for tori}

To prove Theorem~\ref{torusthm}, we will first establish a simpler variant of the first
part of the theorem, \eqref{1.15}, under the stronger assumption that
$V\in L^2(M)$, $V^{-}\in{\mathcal K}(M)$.  After presenting this model argument, we shall see how we can modify
the argument to prove Theorem~\ref{torusthm}.  In all cases,
%
the
main strategy is the same as in the proof of \eqref{1.13} and \eqref{1.14}. That is, we 
need to utilize the standard known result when $V\equiv 1$.  To that end, let us recall the following:
\begin{proposition}\label{standtorus}
If  ${\mathbb T}^n=\Rn/{\mathbb Z}^n$ denotes the standard torus with the flat metric and $N^0(\la)$ denotes the Weyl counting function for $H^0$, then
\begin{equation}\label{4.19}
N^0(\la)=(2\pi)^{-n}\omega_n \mathrm{Vol}_g(M) \, \la^n
+O(\la^{n-1-\frac{n-1}{n+1}}).
\end{equation}
\end{proposition}
The result in \eqref{4.19} is due to Hlwaka \cite{Hlawka} for any $n\geq 2$, which means that \eqref{a.1} holds with $\e=\la^{-\frac{n-1}{n+1}}$ in the case of standard torus. Also
\eqref{4.19} implies that 
\begin{equation}\label{4.19'}
\int_M \sum_{\la_j\in [\la,\la+\e)}|e_j^0(x)|^2dx=
O(\la^{n-1-\frac{n-1}{n+1}}), \,\,\,\forall\,\,\la\ge1,
\end{equation}
which is a special case of \eqref{a.3} with $\e=\la^{-\frac{n-1}{n+1}}$.

 However, 
to obtain the desired bounds for as in Theorem~\ref{torusthm}, we need to modify the earlier arguments since the reminder terms on
right sides of \eqref{a.2} are too large if $\e=\la^{-\frac{n-1}{n+1}}$.  To do this, we begin with the following Proposition.

 \begin{proposition}\label{torus1}
 Let $a=\frac{n-1}{n+1}$, ${\mathbb T}^n=\Rn/{\mathbb Z}^n$ denote the standard torus with the flat metric and $\{e_{\tau_k}\}$ be eigenfunctions of the operator $H_V$, with $V\in L^2(M)$, $V^{-}\in{\mathcal K}(M)$.  Given \eqref{4.19}, if we assume that for all $\la>1$,
\begin{equation}\label{4.24}
\int_M \sum\limits_{\tau_k\in[\la,\la+\la^{-a}]}|e_{\tau_k}(x)|^2 dx \le L\,\la^{n-1-b},
\end{equation}
for some constants $L$ and $b$ with $-1\leq b\leq a$. Then,
\begin{equation}\label{4.25}
\int_M \sum\limits_{\tau_k\in[\la,\la+\la^{-a}]}|e_{\tau_k}(x)|^2 dx \le C_VL\, \la^{n-2+\frac{a-b}{2}} \log \la+C\la^{n-1-a}.
\end{equation}
 \end{proposition}
 
To put this in perspective, note that the conclusion here is reminiscent to how we used \eqref{2.36} to prove the
 $O(\la^{n-\frac32})$ error bounds in \eqref{3.1}.  Indeed, by \eqref{2.35} we have \eqref{4.24} with
 $a=0$ and $b=-1$, and, in this case, the first $\la$-factor in the right side of \eqref{4.25} is
 $\la^{n-2+\frac{a-b}2}=\la^{n-\frac32}$. \eqref{4.25} also represents an improvement over the error bounds in \eqref{3.1} since if we assume $\e=\la^{-a}$, the power on $\e$ on the right side of \eqref{4.25} 
 is $-1/2$ instead of $-1$.
 
 Here the constant $C_V$ in \eqref{4.25} depends on $\|V\|_{L^2(M)}$, which is necessary in the proof of Lemma~\ref{lemma4.7} below. It also depends the implicit constant in \eqref{2.34}, since the proof of \eqref{4.25} will use heat kernel estimates involving $H_V$.
 
Before proving this Proposition, let us present a simple but useful corollary.
 We note that, by 
\eqref{2.36}, if $V\in L^1(M)$, $V^{-}\in{\mathcal K}(M)$,
then for all $x\in M$
\begin{equation}\label{4.26}
 \sum\limits_{\tau_k\in[\la,\la+\la^{-a}]}|e_{\tau_k}(x)|^2 \le  \sum\limits_{\tau_k\in[\la, 2\la]}|e_{\tau_k}(x)|^2 \le C_V\la^n.
\end{equation}
So \eqref{4.24} is true for $b=-1$, and note that every time we apply the Proposition, we would have \eqref{4.24} for a larger value of $b$. 
Consequently, we can obtain the following: 

\begin{corr}\label{err}
Let $V\in L^2(M)$, $V^{-}\in{\mathcal K}(M)$, ${\mathbb T}^n=\Rn/{\mathbb Z}^n$ denotes the standard torus with the flat metric, and $\{e_{\tau_k}\}$ are eigenfunctions of the operator $H_V$. Then given \eqref{4.19}, we have for all $\la>1$
\begin{equation}\label{4.27}
\int_M \sum\limits_{\tau_k\in[\la,\la+\la^{-a}]}|e_{\tau_k}(x)|^2 dx \le 
C_V \la^{n-1-a},
\end{equation}
where $C_V$ is a constant depend on $V$, and $a=\frac{n-1}{n+1}$.
\end{corr}
\begin{proof}
To prove \eqref{4.27}, let us first ignore the $\log\la$  factor on the right side of \eqref{4.25}. Define $b_m$ to be the best exponent such that
$$
\int_M \sum\limits_{\tau_k\in[\la,\la+\la^{-a}]}|e_{\tau_k}(x)|^2 dx \le 
C_V \la^{n-1-b_m},
$$
after applying Proposition~\ref{torus1} $m$ times. We have
\begin{equation}\label{4.27'}\tag{4.6$'$}
 n-1-b_{m+1}=\max\{n-2+\frac{n-1}{2(n+1)}-\frac{b_m}{2},\,\,n-1-\frac{n-1}{n+1}\}, \,\,\,\, m=0,1,2...
\end{equation}
with $b_0=-1$.

Now if $b_m\leq \frac{n-5}{n+1}$, we have $b_{m+1}=\frac{b_m}{2}+\frac{n+3}{2(n+1)} $. In this case, 
 $b_{m+1}-b_{m}=\frac{n+3}{2(n+1)}-\frac{b_m}{2}\geq \frac{4}{n+1}$, which means the sequence is strictly increasing in this case. 
Let $N=[\frac{\frac{n-5}{n+1}+1}{\frac{4}{n+1}}]+1$, we have $b_N>\frac{n-5}{n+1}$. Thus by \eqref{4.27'}, $b_m\equiv \frac{n-1}{n+1}$ for all $m>N$.

Since $\log\la\lesssim \la^\e$ for all $\e>0$, by this argument, we have $b_{N+1}\geq \frac{n-1}{n+1}-\e$. However, if $\e$ is small enough,
$$\max\{n-2+\frac{n-1}{2(n+1)}-\frac{\frac{n-1}{n+1}-\e}{2},\,\,n-1-\frac{n-1}{n+1}\}=n-1-\frac{n-1}{n+1}.
$$
So we have in this case $b_m\equiv \frac{n-1}{n+1}$ for all $m>N+1$. The proof of \eqref{4.27} is complete.

\end{proof}

By using Corollary \ref{err}, along with the arguments in Section 3, we have the following result on torus.
\begin{theorem}\label{torusthm1}
If ${\mathbb T}^n=\Rn/{\mathbb Z}^n$ denotes the standard torus with the flat metric and $V\in L^2(M)$, $V^{-}\in{\mathcal K}(M)$, then given \eqref{4.19}, we have 
\begin{equation}\label{4.28}
N_V(\la)=(2\pi)^{-n}\omega_n \mathrm{Vol}_g(M) \, \la^n
+O(\la^{n-1-\frac{n-1}{n+1}}).
\end{equation}
\end{theorem}

Note that
\eqref{4.28} is a variant of \eqref{1.15} with a stronger condition on the potential $V$. The condition $V\in L^2(M)$ arises when we try to get an improvement 
over the main terms in \eqref{3.1}, which come from applying Lemma~\ref{lemma3.3} before.  For more details, see Lemma~\ref{lemma4.7} in the argument below, and we postpone the proof of \eqref{1.15} to the end of the section.

\begin{proof}[Proof of Proposition \ref{torus1} and Theorem \ref{torusthm1}] We shall first give the proof of \eqref{4.28}, then prove \eqref{4.25} by modifying the
argument.

Let $\ala(\tau)$ be defined as in \eqref{2.15} with $\e=\la^{-\frac{n-1}{n+1}}$. In view of \eqref{4.19} and \eqref{4.27} by repeating the arguments in \eqref{2.13}-\eqref{2.23}
we see that \eqref{4.28} would be a consequence of
\begin{multline}\label{4.34}
\Bigl| \, \sum_{j,k}\int_M \int_M
\frac{\ala(\la_j)-\ala(\tau_k)}{\la_j^2-\tau_k^2} e_j^0(x)e_j^0(y)V(y)e_{\tau_k}(x)
e_{\tau_k}(y) \, dx dy\, \Bigr| 
\le 
C_V \la^{n-2}\log \la,
\end{multline}
where we interpret 
\begin{equation}\label{4.35}
\frac{\ala(\tau)-\ala(\mu)}{\tau^2-\mu^2}=\frac{\ala^\prime(\tau)}{2\tau},  \hspace{2mm} \text{if}\,\, \,\,\tau=\mu.
\end{equation}

The proof of \eqref{4.34} requires a bound on the trace of certain spectral projection operators, which is \eqref{4.27}. 
The fact that we rely on trace inequalities, rather than pointwise ones as was done in the past, accounts for our assumption here that $V\in L^2(M)$.

As before, we shall split things into three different cases that require slightly different arguments. The main contribution still comes from frequencies $\tau_k$ which are comparable to $\la$. The proof for large or small frequencies $\tau_k$ directly follow from applying Proposition~\ref{largeprop} and Proposition~\ref{smallprop}, 
and these two cases only contribute terms which are $O( \|V\|_{L^1(M)}\la^{n-2}(\log \la)^{1/2})$.
%

Consequently, we would obtain \eqref{4.34} if we could prove the following.

\begin{proposition}\label{compprop1}
As in Theorem~\ref{torusthm1}, fix $V\in L^2(M)$ with $V^{-}\in{\mathcal K}(M)$. If  $\ala(\tau)$ is defined as in \eqref{2.15} with $\e=\la^{-\frac{n-1}{n+1}}$ , then
\begin{multline}\label{4.36}
\Bigl| \, \sum_{j}
\sum_{\{k: \, \tau_k\in [\lambda/2,10\lambda]\}}\int_M \int_M
\frac{\ala(\la_j)-\ala(\tau_k)}{\la_j^2-\tau_k^2} e_j^0(x)e_j^0(y)V(y)e_{\tau_k}(x)
e_{\tau_k}(y) \, dx dy\, \Bigr| 
\\
\le 
C_V \la^{n-2}\log \la.
\end{multline}
\end{proposition}

To prove Proposition~\ref{compprop1}, we shall follow the same setup as in \eqref{3.2}-\eqref{3.5}, i.e. we need to show that
\begin{equation}\label{c.1}\tag{4.10$'$}
\Bigl| \sum_{\tau_k \in [\la/2,10\la]}\iint
K_{\tau_k}(x,y)e_{\tau_k}(x)e_{\tau_k}(y) \, V(y)\, dxdy\Bigr|
\le C_V  \la^{n-2}\log \la,
\end{equation}
where we shall write $K_{\tau}(x,y)$ as in \eqref{3.3} for $\tau=\tau_k\in [\la/2,\la]$ and \eqref{3.4}
for $\tau=\tau_k\in (\la,10\la]$ separately.

We just need to prove \eqref{c.1} when $K_{\tau}(x,y)$ is replaced by $K_{\tau, \ell_0}(x,y)$, $ \sum_{\ell} K^-_{\tau, \ell}(x,y)$, and $\sum_{\ell} R_{\tau, \ell}(x,y)(1-\ala(\tau))$ 
 if $\tau=\tau_k\in [\la/2,\la]$, and when $K_{\tau}(x,y)$ is replaced by $K_{\tau, \ell_0}(x,y)$, $ \sum_{\ell} K^+_{\tau, \ell}(x,y)$, and $\sum_{\ell} R_{\tau, \ell}(x,y)\ala(\tau)$ if $\tau=\tau_k\in (\la,10\la]$,
 since as in the proof of Proposition~\ref{compprop}, the other terms only contribute terms which are $O( \|V\|_{L^1(M)}\la^{n-2})$ to the right side of \eqref{c.1}.

To proceed, we need the following lemma which is essentially a refined version of Lemma~\ref{lemma3.2}.
\begin{lemma}\label{lemma4.6}
If $a=\frac{n-1}{n+1}$, $\ell \in {\mathbb Z}$, $\la^{-a}<2^\ell\le \la/100$,
and $j=0,1,2,\dots$, we have for each $N\in {\mathbb N}$
\begin{multline}\label{4.41}
\|K^{\pm}_{\tau,\ell}(\, \cdot\, ,y )\|_{L^2(M)}, \, \,
\|2^\ell \tfrac\partial{\partial\tau}K^{\pm}_{\tau,\ell}(\, \cdot\, ,y )\|_{L^2(M)}
\\
\lesssim \la^{\frac{n-1}2-1}2^{-\ell/2}(1+j)^{-N},
\quad \tau\in I_{\ell,j}^\pm \cap [\la/2,10\la].
\end{multline}
Also, 
\begin{multline}\label{4.42}
\|K_{\tau,\ell_0}(\, \cdot\, ,y )\|_{L^2(M)}, \, \, 
\|\la^{-a}\tfrac\partial{\partial\tau}K_{\tau,\ell_0}(\, \cdot\, ,y )\|_{L^2(M)} 
\\
\lesssim
\la^{\frac{n-1+a}2-1}(1+j)^{-N}, \quad 
\tau\in I_{\ell_0,j}^\pm \cap [\la/2,10\la],
\end{multline}
Finally, we also have for $\la^{-a}< 2^\ell \le \la/100$ and $\tau\in [\la/2,10\la]$
\begin{equation}\label{4.45}
\|R_{\tau,\ell}(\, \cdot\, ,y )\|_{L^2(M)}, \, \, 
\|2^\ell \tfrac\partial{\partial \tau}R_{\tau,\ell}(\, \cdot\, ,y)\|_{L^2(M)}
\lesssim \la^{\frac{n-1}2-1} 2^{-\ell/2}.
\end{equation}
\end{lemma}

As before, we are abusing notation a bit.  First, in \eqref{4.41}
we mean that if $K_{\tau,\ell}$ equals $K^+_{\tau,\ell}$ or
$K^-_{\tau,\ell}$ then the bounds in \eqref{4.41} for
$\tau$ in $I_{\ell,j}^+ \cap [\la,10\la]$ or
$I_{\ell,j}^- \cap [\la/2,\la]$, respectively. 
\begin{proof}\renewcommand{\qedsymbol}{}
To prove the first inequality we note that if $\tau\in I^\pm_{\ell,j}\cap[\la/2,10\la]$ and $\beta(2^{-\ell}(\la_i-\tau))\ne 0$, then
$|\la_i-\tau|\le 2^{\ell+1}$ and $\la_i,\tau\approx \la$, and, in this case,
we also have $\ala(\la_i)-1=O((1+|j|)^{-N})$ if $\tau\in I^-_{\ell,j}$ and $\ala(\la_i)=O((1+|j|)^{-N})$ if
$\tau\in I^+_{\ell,j}$.  Therefore, by \eqref{b.5}, we have
\begin{align*}
\bigl\|K^\pm_{\tau,\ell}(\, \cdot \, ,y)\bigr\|_{L^2(M)}&\lesssim
(1+|j|)^{-N}\la^{-1}2^{-\ell} \, \bigl(\sum_{\{i: \, |\la_i-\tau|\le 2^{\ell+1}\}} |e^0_i(y)|^2\bigr)^{1/2}
\\
&\le (1+|j|)^{-N}2^{-\ell/2}\la^{\frac{n-1}2-1},
\end{align*}
which is the first part of \eqref{4.41}.  In the second inequality, we used the fact that  
\begin{equation}\label{4.55}
\sum\limits_{\la_i\in[\la,\la+2^\ell]}|e^0_i(x)e^0_i(y)|\lesssim \la^{n-1}2^\ell, \,\,\,\forall\,\, 2^\ell\ge \la^{-a},
\end{equation}
which is a consequence of \eqref{4.19'} if we choose$ \{e_i^0\}$ to be the standard orthonormal basis,
$\{\exp(2\pi ij\cdot x,\,\, j \in \mathbb{Z}_n)\}$ for the Laplacian on the torus.  For we then have that
the left side of \eqref{4.55} equals the number of eigenvalues of $P^0$ in $[\la, \la + \la^{-a}]$.

The other inequality
in \eqref{4.41} follows from this argument since
$$\frac\partial{\partial \tau}
\frac{\tilde \beta(2^{-\ell}(\la_i-\tau))}{\la_i+\tau}=O(\la^{-1}2^{-\ell}),
$$
due to the fact that we are assuming that $2^\ell \le \la/100$.

This argument also gives us \eqref{4.42} if we use the fact that
$\tau\to (\ala(\tau)-\ala(\mu))/(\tau^2-\mu^2)$ is smooth if we define it as in \eqref{4.35} when
$\tau=\mu$ and use the fact that
$$\partial_\tau^k \bigl(\beta_0(\la_i-\tau) (\ala(\la_i)-\ala(\tau))/(\la_i-\tau)\bigr)=O(\la^{a(k+1)}(1+|j|)^{-N}), \, \, k=0,1,
\, \, \tau\in I_{0,j}^\pm,$$
and the fact that, if this expression is nonzero, we must have $|\la_i-\tau|\le 2\la^{-a}$.

The remaining estimates in \eqref{4.45} just follows from the same argument as in the proof of \eqref{4.41}.
\end{proof}
We shall also need the following lemma for the proof of Proposition~\ref{compprop}.
\begin{lemma}\label{lemma4.7}
Let $a=\frac{n-1}{n+1}$, $I=[a_0, a_0+\gamma]$ be an interval of length $\la^{-a}\le\gamma \le \la$, and assume that for any fixed $\tau \in I \cap [\la/2,10\la]$, $w_\tau(x,y)\in C^1(\mathbb{R}\times M\times M)$ satisfies
\begin{equation}\label{c.2}
\|w_{\tau}(\, \cdot\, ,y )\|_{L^2(M)}, \, \,
\|\gamma \tfrac\partial{\partial\tau}w_{\tau}(\, \cdot\, ,y )\|_{L^2(M)}
\le L,
\end{equation}
for some constant $L$.
Then if $\beta\in C^\infty(\mathbb{R})$ and $V\in L^2(M)$, we have
\begin{multline}\label{c.3}
\Bigl| \sum_{\tau_k\in I\cap[\la/2,10\la]}
\iint w_{\tau_k}(x,y)\beta(\tau_k)e_{\tau_k}(x)e_{\tau_k}(y)
V(y)\, dydx\Bigr| \\
\lesssim  \sup_{\tau\in I\cap[\la/2,10\la]}|\beta(\tau)|\cdot  L\,\la^{\frac{n-1}{2}}\gamma^{\frac12}.
\end{multline}
\end{lemma}
\begin{proof}
For any fixed $y\in M$, by applying Lemma~\ref{delta} with
$\delta=\gamma$, $m(\tau, x)=w_{\tau+a}(x,y)$ and $a_k=\beta(\tau_k)e_{\tau_k}(y)$, we have
\begin{align*}
\Bigl| &\sum_{\tau_k\in I\cap[\la/2,10\la]}
\iint w_{\tau_k}(x,y)\beta(\tau_k)e_{\tau_k}(x)e_{\tau_k}(y)
V(y)\, dydx\Bigr|
\\
&\le |V\|_{L^2}
\cdot  \Bigl\|
\sum_{\tau_k\in I\cap[\la/2,10\la]}
\iint w_{\tau_k}(x,y)e_{\tau_k}(x)e_{\tau_k}(y)
\Bigr\|_{L^2\big(dy; L^1(dx)\big)}
\\
&\le \|V\|_{L^2} \cdot \sup_y
\Bigl( \|w_{a_0}(\, \cdot \, ,y)\|_{L^2(M)}+\int^\gamma_{0}\bigl\|\tfrac\partial{\partial \tau}w_{s+a_0}(\, \cdot \, y)\|_{L^2(M)} \, ds
\Bigr) \notag
\\
&\qquad \qquad \qquad \qquad \qquad \times \bigl(\int_M\sum_{\tau_k\in I\cap [\la/2,10\la]}
|\beta(\tau_k)e_{\tau_k}(y)|^2 dy\bigr)^{1/2}
\notag
\\
&\lesssim \sup_{\tau\in I\cap[\la/2,10\la]}|\beta(\tau)|\cdot  L\,\la^{\frac{n-1}{2}}\gamma^{\frac12}.
\notag
\end{align*}
In the second to last inequality we used \eqref{c.2} and the fact that, 
by Corollary \ref{err} , 
\begin{equation}\label{3.13}
\sum_{\tau_k\in I\cap [\la/2,10\la]} |e_{\tau_k}(y)|^2 \lesssim
 \gamma\,\la^{n-1}, \,\,\forall\,\,\, \la\ge 1\,\,\,\text{if}\,\,\,\la^{-a}\le\gamma \le \la.
\end{equation}
As we alluded to before, since Corollary~\ref{err} only affords us trace bounds, 
the preceding inequality involves $\|V\|_{L^2(M)}$ in the right, as opposed to 
$L^1$-norms of the potential as was the case in the past.
\end{proof}

\begin{proof}[Proof of Proposition~\ref{compprop1}]
First, if we apply Lemma~\ref{lemma4.7} with $w_\tau(x,y)=K^{\pm}_{\tau,\ell}(x,y )$, $\gamma=2^\ell$ and $\beta(\tau)\equiv 1$, by \eqref{4.41} 
\begin{equation}\label{c.5}
\Bigl| \sum_{\tau_k\in I^\pm_{\ell,j}\cap (\la,10\la]}\iint K^\pm_{\tau_k,\ell}(x,y)e_{\tau_k}(x)e_{\tau_k}(y)
V(y)\, dydx\Bigr|  \lesssim  \la^{n-2} \cdot (1+j)^{-N}
\end{equation}
If we sum over $j=0,1,2,\dots$, we see that for $\la^{-a}< 2^\ell \le \la/100$, \eqref{c.5} yields
\begin{multline}\label{c.6}
\Big| \sum_{\la <\tau_k\le 10\la}
\iint K^+_{\tau_k,\ell}(x,y) e_{\tau_k}(x)e_{\tau_k}(y) 
\, V(y)\, dxdy\Bigr|
\\
+\Big| \sum_{\la/2 \le \tau_k\le\la}
\iint K^-_{\tau_k,\ell}(x,y) e_{\tau_k}(x)e_{\tau_k}(y) 
\, V(y)\, dxdy\Bigr|
\lesssim \la^{n-2}.
\end{multline}
If we take $w_\tau(x,y)=K_{\tau,\ell_0}(x,y )$, $\gamma=\la^{-a}$ and $\beta(\tau)\equiv 1$ in Lemma~\ref{lemma4.7}, this argument
also gives
\begin{multline}\label{c.7}
\Bigl| \sum_{\la/2 \le \tau_k\le \la}
\iint K_{\tau_k,\ell_0}(x,y) e_{\tau_k}(x)e_{\tau_k}(y) \, 
V(y)\, dx dy\Bigr|
\\
+ \Bigl| \sum_{\la < \tau_k\le 10\la}
\iint K_{\tau_k,\ell_0}(x,y) e_{\tau_k}(x)e_{\tau_k}(y) \, 
V(y)\, dx dy\Bigr|
\lesssim \la^{n-2}.
\end{multline}

Next, since $R_{\tau,\ell}$ enjoys the bounds in \eqref{4.45}, we can use Lemma~\ref{lemma4.7} with $w_\tau(x,y)=R_{\tau,\ell}(x,y )$, $\gamma=2^\ell$ and $\beta(\tau)=\ala(\tau)$ 
to see that for $\la^{-a}< 2^\ell \le \la/100$ we
have
\begin{align*}
\Bigl| \sum_{\tau_k\in I^+_{\ell,j}\cap (\la,10\la]}
&\iint R_{\tau_k,\ell}(x,y)\ala(\tau_k)e_{\tau_k}(x)
e_{\tau_k}(y) V(y)\, dxdy\Bigr|
\\
&\lesssim \|V\|_{L^1}\cdot 2^{-\ell/2}\la^{\frac{n-1}2-1}
\sup_y\bigl(\sum_{\tau_k\in I^+_{\ell,j}\cap
(\la,10\la]} |\ala(\tau_k)e_{\tau_k}(y)|^2\bigr)^{1/2} 
\\
&\lesssim \la^{n-2}\cdot (1+j)^{-N},
\end{align*}
since $\ala(\tau_k)=O((1+j)^{-N})$ if $\tau_k
\in I^+_{\ell,j}$.  Summing over this bound over $j$
of course yields
\begin{equation}\label{c.8}
\Bigl|\sum_{\la<\tau_k\le 10\la}
\iint R_{\tau_k,\ell}(x,y)
\ala(\tau_k)e_{\tau_k}(x)e_{\tau_k}(y) V(y)\, dxdy\Bigr|
\lesssim \la^{n-2}.
\end{equation}
The same argument gives
\begin{equation}\label{c.9}
\Bigl|\sum_{\la/2\le\tau_k\le \la}
\iint R_{\tau_k,\ell}(x,y)
\bigl(1-\ala(\tau_k)\bigr)e_{\tau_k}(x)e_{\tau_k}(y) V(y)\, dxdy\Bigr|
 \lesssim \la^{n-2}.
\end{equation}

We now have assembled all the ingredients for the proof
of \eqref{c.1}.  If we use \eqref{c.6}, \eqref{c.7},
\eqref{c.9}, \eqref{3.19} and \eqref{3.20} along with
\eqref{3.3}, we conclude that the analog of \eqref{c.1}
must be valid where the sum is taken over
$\tau_k\in [\la/2,\la]$.  The log-loss comes from the fact
that there are $\approx \log \la$ terms
$K^-_{\tau,\ell}$ and $R_{\tau,\ell}$.
We similarly obtain the analog of \eqref{c.1} where the sum is taken over $\tau_k\in (\la,10\la]$ from \eqref{3.4} along with
\eqref{c.6}, \eqref{c.7}, \eqref{3.16}, \eqref{c.8} and \eqref{3.19}.

From this, we deduce that \eqref{c.1} must be valid,
which finishes the proof of Proposition~\ref{compprop1}.
\end{proof}

\end{proof}
Now we give the proof of Propostion~\ref{torus1}. 
Let $a=\frac{n-1}{n+1}$ and $\tilde \chi_\la(\tau)$ be defined as in \eqref{4.5'} with $\e=\la^{-\frac{n-1}{n+1}}$. In view of \eqref{4.19} and \eqref{4.27} by repeating the arguments in \eqref{4.2}-\eqref{4.8}
we see that \eqref{4.25} would be a consequence of
\begin{multline}\label{d.1}
\Bigl| \, \sum_{j,k}\int_M \int_M
\frac{\tilde \chi_\la(\la_j)-\tilde \chi_\la(\tau_k)}{\la_j^2-\tau_k^2} e_j^0(x)e_j^0(y)V(y)e_{\tau_k}(x)
e_{\tau_k}(y) \, dx dy \Bigr|
\\
\le C_{V}  \, 
\la^{n-2+\frac{a-b}{2}}\log \la,
\end{multline}
where we interpret 
\begin{equation}
\frac{\tilde \chi_\la(\tau)-\tilde \chi_\la(\mu)}{\tau^2-\mu^2}=\frac{\tilde \chi_\la^\prime(\tau)}{2\tau},  \hspace{2mm} \text{if}\,\, \,\,\tau=\mu.
\nonumber
\end{equation}

As before, we shall split things into three different cases that require slightly different arguments. The main contribution still comes from frequencies $\tau_k$ which are comparable to $\la$. For large or small frequencies $\tau_k$, the proof is completely analogous to the proof of Proposition~\ref{largeprop} and Proposition~\ref{smallprop}, and these two cases only contribute terms to the right side of \eqref{d.1} which are
$O( \|V\|_{L^1(M)}\la^{n-2}(\log \la)^{1/2})$.
%

Conequently, we would obtain \eqref{d.1} if we could prove the following.
\begin{multline}\label{4.102}
\Bigl| \, \sum_{j}
\sum_{\{k: \, \tau_k\in [\lambda/2,10\lambda]\}}\sum_{j,k}\int_M \int_M
\frac{\tilde \chi_\la(\la_j)-\tilde \chi_\la(\tau_k)}{\la_j^2-\tau_k^2} e_j^0(x)e_j^0(y)V(y)e_{\tau_k}(x)
e_{\tau_k}(y) \, dx dy \Bigr|
\\
\le C_{V} \, 
\la^{n-2+\frac{a-b}{2}}\log \la.
\end{multline}
The proof of \eqref{4.102} is similar to \eqref{4.36}. After replacing $\ala(\tau)$ by $\tilde \chi_\la(\tau)$ in the proof of \eqref{4.36}, the 
main contribution, which gives the right side of \eqref{4.102}, still comes from terms where Lemma~\ref{lemma4.7} is applied, i.e., \eqref{c.5}-\eqref{c.9}. The difference 
is we use \eqref{4.24} instead of Corollary~\ref{err} in the proof of Lemma~\ref{lemma4.7}. Moreover,  we do not 
need to divide the proof into two cases as in \eqref{3.3} and \eqref{3.4}, since 
$\tilde \chi_\la(\tau)$ is rapidly decreasing away from $\la$ on both regions $\tau_k\in[\la/2, \la] $ and $\tau_k\in[\la, 10\la] $. We skip the 
proof of \eqref{4.102} here for simplicity.

\quad\newline

Having presented the model argument, let us now prove Theorem~\ref{torusthm}.

\begin{proof}[Proof of \eqref{1.16}]
To get an improvement over the error term as in \eqref{1.16},  we need the following

\begin{proposition}\label{standtorus1}
If   ${\mathbb T}^n=\Rn/{\mathbb Z}^n$ denotes the standard torus with the flat metric and $N^0(\la)$ denotes the Weyl counting function for $H^0$, then
\begin{equation}\label{4.56}
N^0(\la)=(2\pi)^{-n}\omega_n \mathrm{Vol}_g(M) \, \la^n
+r_n(\la),
\end{equation}
where
$$r_n(\la)\lesssim\begin{cases}
\la^{n-2} , \qquad\qquad\,\,\,\,\, \text{if} \,\,\, n\geq 5 \\
\la^{2}(\log \la)^{2/3} ,\quad \,\,\,\,\,\, \text{if} \,\,\, n=4 \\
\la^{\frac{21}{16}+\e} , \quad\qquad\quad\,\,\,\, \text{if} \,\,\, n=3 \\
\la^{\frac{131}{208}}(\log \la)^{\frac{18627}{8320}}  , \,\,\, \text{if} \,\,\, n=2.\\
\end{cases}
$$
\end{proposition}
There has been a lot of research related to the Weyl formula on the  torus, which is equivalent to counting the lattice points inside the ball of radius $\la$.
Currently, the exact order of the error term is only known when $n\geq 5$.  See e.g. E. Landau \cite{landau3vorlesungen}, A. Walfisz \cite{walfisz1957gitterpunkte}, and E. Kr\"atzel \cite{kratzel2013analytische}. The above best known results in lower dimensions are due to A. Walfisz \cite{walfisz1959gitterpunkte} ($n$=4), D. R. Health-Brown \cite{heath1999lattice} ($n$=3), and M. N. Huxley \cite{huxley2003exponential} ($n$=2). 
For  more details and a discussion of recent progress 
on the problem, see e.g. the survey paper \cite{ivic2004lattice}, and W. Freeden \cite{freeden2011metaharmonic}.

For simplicity, we will only give the proof of \eqref{1.16} for $n\geq 5$.   The proof for  $n=4$  follows from the same argument due to the fact that the extra $(\log \la)^{2/3}$-factor is harmless in the presence of the $\la^\e$-factor in \eqref{1.16}. Also, for the $n=2,3$ cases, if we use the improved results in \eqref{4.56}, by the same argument as in 
the proof of Theorem \ref{4.6}, we can recover the improved bound without a $\la^\e$-loss.

As a consequence of \eqref{4.56}, if $n\ge5$,  we also have the following analog of \eqref{4.19'}
\begin{equation}\label{e.1}
\int_M \sum_{\la_j\in [\la,\la+\la^{-1})}|e_j^0(x)|^2dx=
O(\la^{n-2}), \,\,\,\forall\,\,\la\ge1.
\end{equation}

As before, we  first prove a Proposition which allows us to do iterations.
\begin{proposition}\label{torus2}
 Let  ${\mathbb T}^n=\Rn/{\mathbb Z}^n$ denote the standard torus with the flat metric and $\{e_{\tau_k}\}$ be eigenfunctions of the operator $H_V$. Fix $V\in L^2(M)$, $V^{-}\in{\mathcal K}(M)$, given \eqref{4.56}, when $n\geq 5$, if we assume, for all $\la>1$
\begin{equation}\label{4.59}
\int_M \sum\limits_{\tau_k\in[\la,\la+\la^{-1}]}|e_{\tau_k}(x)|^2 dx \le L\,\la^{n-1-b},
\end{equation}
for some constants $L$ and $b$ with $-1\leq b\leq 1$. Then,
\begin{equation}\label{4.60}
\int_M \sum\limits_{\tau_k\in[\la,\la+\la^{-1}]}|e_{\tau_k}(x)|^2 dx \le C_V L\,\la^{n-2+\frac{1-b}{2}} \log \la+C\la^{n-2}.
\end{equation}
 \end{proposition}

Note that, by \eqref{4.26}, \eqref{4.59} holds for $b=-1$,  and  that every time we apply the Proposition, we would have \eqref{4.59} for a larger value of $b$.  Consequently, just as before, after finitely many iterations, we will obtain the following:
\begin{corr}\label{err1}
Let $V\in L^2(M)$, $V^{-}\in{\mathcal K}(M)$, ${\mathbb T}^n=\Rn/{\mathbb Z}^n$ denote the standard torus with the flat metric, and $\{e_{\tau_k}\}$ be eigenfunctions of the operator $H_V$. Then given \eqref{4.56}, when $n\geq 5$, we have for all $\la>1$
\begin{equation}\label{4.61}
\int_M \sum\limits_{\tau_k\in[\la,\la+\la^{-1}]}|e_{\tau_k}(x)|^2 dx \le 
C_{V,\e} \la^{n-2+\e}, \quad \forall \, \e>0,
\end{equation}
where 
$C_{V,\e}$ is a constant depending on $V$ and $\e$.
\end{corr}
\begin{proof}
To prove \eqref{4.61}, let us first ignore the $\log\la$ factor on the right side of \eqref{4.25}. Define $b_m$ to be the best exponent such that
$$
\int_M \sum\limits_{\tau_k\in[\la,\la+\la^{-1}]}|e_{\tau_k}(x)|^2 dx \le 
C_V \la^{n-1-b_m},
$$
after applying Proposition~\ref{torus1} $m$ times. We have
\begin{equation}
 b_{m+1}=\frac{1+b_m}{2} \,\,\,\, m=0,1,2...
 \nonumber
\end{equation}
with $b_0=-1$.

By solving the arithmetic sequences explicitly, we have $b_m=1-\frac{1}{2^{m-1}} \,\,\,\, m=0,1,2...$. So \eqref{4.61} follows 
by letting $m\rightarrow \infty$. And since $\log\la\leq C_\e \la^\e$ for all $\e$, \eqref{4.61} follows from the same argument if we consider the $\log\la$-factor.

\end{proof}

Let $\tilde \chi_\la(\tau)$ be defined as in \eqref{4.5'} with $\e=1/\la$. In view of \eqref{4.19} and \eqref{4.27} by repeating the arguments in \eqref{4.2}-\eqref{4.8}
we see that \eqref{4.60} would be a consequence of
\begin{multline}\label{e.2}
\Bigl| \, \sum_{j,k}\int_M \int_M
\frac{\tilde \chi_\la(\la_j)-\tilde \chi_\la(\tau_k)}{\la_j^2-\tau_k^2} e_j^0(x)e_j^0(y)V(y)e_{\tau_k}(x)
e_{\tau_k}(y) \, dx dy \Bigr|
\\
\le C_{V}  \, 
\la^{n-2+\frac{1-b}{2}}\log \la,
\end{multline}

Similarly, Let $\ala(\tau)$ be defined as in \eqref{2.15} with $\e=1/\la$. In view of \eqref{4.56} and \eqref{4.61} by repeating the arguments in \eqref{2.13}-\eqref{2.23}
we see that \eqref{1.16} would be a consequence of
\begin{multline}\label{e.3}
\Bigl| \, \sum_{j,k}\int_M \int_M
\frac{\ala(\la_j)-\ala(\tau_k)}{\la_j^2-\tau_k^2} e_j^0(x)e_j^0(y)V(y)e_{\tau_k}(x)
e_{\tau_k}(y) \, dx dy\, \Bigr| 
\le 
C_{V,\e} \la^{n-2+\e/2}\log \la.
\end{multline}
Here the number $\e$ on the right side of \eqref{e.3} is the same as the $\e$ appeared in \eqref{4.61}, which can be arbitrary small. And it is irrelevant to the constant $\e$ in the definition of 
$\ala(\tau)$ in \eqref{2.15}, which we choose to be several different numbers throughout this section.

So Proposition \ref{torus2} and \eqref{1.16} follow from \eqref{e.2} and \eqref{e.3}, respectively, and both require 
a bound on the trace of certain spectral projection operators, which are \eqref{4.59} and \eqref{4.61} respectively. 

Similarly, we shall split things into three different cases. The main contribution still comes from frequencies $\tau_k$  
which are comparable to $\la$. For large or small frequencies $\tau_k$, by applying the arguments in the proof of Proposition~\ref{largeprop} and Proposition~\ref{smallprop}, 
, we see that the left side of \eqref{e.2} and \eqref{e.3} would be bounded by  $C_V \, 
\la^{n-2}(\log \la)^{1/2}$, which are better than desired bounds.

Finally, for frequencies  $\tau_k$  which are comparable to $\la$, if we let $a=1$ in the proof of Proposition~\ref{compprop1}, 
and use \eqref{4.59} or \eqref{4.61} correspondingly for the main terms (i.e., results of Lemma~\ref{lemma4.7}.), it follows from 
the same arguments as in \eqref{c.5}-\eqref{c.9} that the left side of \eqref{e.2} and \eqref{e.3} are controlled by
their right sides. The proof of \eqref{1.16} is complete.
\end{proof}

\begin{proof}[Proof of \eqref{1.15}]  To recover to Hlawka bound \cite{Hlawka} under the weaker conditions on $V$ in Theorem~\ref{torusthm},  the strategy is 
similar to previous cases. That is, to get improvements for the main terms in \eqref{3.1}, which are essentially the bounds in Lemma~\ref{lemma4.7} above.
We begin with the following Proposition. 

\begin{proposition}\label{torus12}
 Let  ${\mathbb T}^n=\Rn/{\mathbb Z}^n$ denote the standard torus with the flat metric, and $\{e_{\tau_k}\}$ be the eigenfunctions of the operator $H_V$. Fix $V\in L^p(M)$, $V^{-}\in{\mathcal K}(M)$ for some $p>\frac{2n}{n+2}$ for $n\ge3$ and $V\in \mathcal{K}(M)$ for $n=2$,
given \eqref{4.19},
  if we assume, for all $\la>1$ that
\begin{equation}\label{4.69}
\int_M \sum\limits_{\tau_k\in[\la,\la+\la^{-a}]}|e_{\tau_k}(x)|^2 dx \le L\,\la^{n-1-b},
\end{equation}
for some constants $L$ and $ b$ with $-1\leq b\leq a$, then it follows that for any $1\leq p\leq 2$ we have
\begin{equation}\label{4.70}
\int_M \sum\limits_{\tau_k\in[\la,\la+\la^{-a}]}|e_{\tau_k}(x)|^2 dx \le \begin{cases} C_V L \la^{\frac16-\frac b2} \log \la+C\la^{2/3} \,\,\,\, \text{if} \,\,\, n=2\\
C_V L \la^{k(b,p)} \log \la+C\la^{n-1-a}  \,\,\,\, \text{if} \,\,\, n\geq 3,
\end{cases}
\end{equation}
where $a=\frac{n-1}{n+1}$, and $k(b,p)=\frac{n-1+a}{2}-1+\frac{n-1-b}{2}\cdot (2-\frac2p)+\frac n2\cdot(\frac2p-1)$.
 \end{proposition}
 Here the constant $C_V$ in \eqref{4.25} depends on $\|V\|_{L^1(M)}$ if $n=2$ and $\|V\|_{L^p(M)}$ if $n\ge 3$. It also depends on the implicit constants in \eqref{2.34}, since the proof of \eqref{4.25} will use heat kernel estimates involving $H_V$.

If $V\in L^1(M)$, $V^{-}\in{\mathcal K}(M)$, by \eqref{4.26}, \eqref{4.69} is true for $b=-1$.  In particular, when $n=2$, if $V\in{\mathcal K}(M)$, by applying the spectral projection bounds in \cite{BSS}, we have
\begin{equation}\label{4.71}
 \sum\limits_{\tau_k\in[\la,\la+\la^{-\frac13}]}|e_{\tau_k}(x)|^2 \le  \sum\limits_{\tau_k\in[\la,\la+1]}|e_{\tau_k}(x)|^2 \le C_V\la.
\end{equation}
So \eqref{4.69} is true for $b=0$ when $n=2$.  As before, after finitely many iterations, we have:
\begin{corr}\label{err2}
Let ${\mathbb T}^n=\Rn/{\mathbb Z}^n$ denote the standard torus with the flat metric, and $\{e_{\tau_k}\}$ be the eigenfunctions of the operator $H_V$. 
Assume also that $V\in L^p(M)$, $V^{-}\in{\mathcal K}(M)$ for some $p>\frac{2n}{n+2}$ for $n\ge3$ and $V\in \mathcal{K}(M)$ for $n=2$. 
Then given \eqref{4.19}, we have for all $\la>1$
\begin{equation}\label{4.72}
\int_M \sum\limits_{\tau_k\in[\la,\la+\la^{-a}]}|e_{\tau_k}(x)|^2 dx \le 
C_V \la^{n-1-a},
\end{equation}
where $C_V$ is a constant depending on $V$, and $a=\frac{n-1}{n+1}$.
\end{corr}

\begin{proof}
For $n=2$, by using \eqref{4.71}, which is \eqref{4.69} corresponding to $b=0$, \eqref{4.72} follows from \eqref{4.70} directly  since $C_V\|V\|_{L^1(M)} \la^{\frac16-\frac b2} \log \la$ is better than the right side of \eqref{4.72}.

To prove \eqref{4.72} for $n\geq 3$, let us first ignore the $\log\la$  factor on the right side of \eqref{4.70}. Define $b_m$ to be the best exponent such that
$$
\int_M \sum\limits_{\tau_k\in[\la,\la+\la^{-a}]}|e_{\tau_k}(x)|^2 dx \le 
C_V \la^{n-1-b_m},
$$
after applying Proposition~\ref{torus12} $m$ times. We have
\begin{equation}\label{4.75'}\tag{4.75$'$}
 n-1-b_{m+1}=\max\{k(b_m,p),\,\,n-1-\frac{n-1}{n+1}\}, \,\,\,\, m=0,1,2...
\end{equation}
with, as before, $b_0=-1$.

Now by a straightforward calculation, if $b_m\leq \frac{2n-(n+2)p}{(n+1)(p-1)}+\frac{n-1}{n+1}$, we have $b_{m+1}=n-1-k(b_m,p) $. In this case, 
 $b_{m+1}-b_{m}=\frac{n+3}{2(n+1)}-\frac1p+\frac12-\frac{b_m}{p}\geq \mu(p)$, where
 $$\mu(p)=\frac{n+3}{2(n+1)}-\frac1p+\frac12-\frac{2n-(n+2)p}{(n+1)(p-1)p}-\frac{n-1}{(n+1)p}>0,\,\,\,\, \text{if}\,\,\,\, p>\frac{2n}{n+2}.
 $$
 So the sequence is strictly increasing in this case. 
If $N=[\frac{\frac{2n-(n+2)p}{(n+1)(p-1)}+\frac{n-1}{n+1}+1}{\mu(p)}]+1$, we have $b_N>\frac{2n-(n+2)p}{(n+1)(p-1)}+\frac{n-1}{n+1}$. Thus by \eqref{4.75'}, $b_m\equiv \frac{n-1}{n+1}$ for all $m>N$.

Since $\log\la\lesssim \la^\e$ for all $\e$, by the same argument, we have $b_{N+1}\geq \frac{n-1}{n+1}-\e$. However, if $\e$ is small enough,
$$\max\{k(\frac{n-1}{n+1}-\e,p),\,\,n-1-\frac{n-1}{n+1}\}=n-1-\frac{n-1}{n+1}, \,\,\,\, \text{if}\,\,\,\, p>\frac{2n}{n+2}.
$$
So we have in this case $b_m\equiv \frac{n-1}{n+1}$ for all $m>N+1$. The proof of \eqref{4.72} is complete.

\end{proof}

To obtain \eqref{4.70}, if we repeat the arguments in the proof of Proposition~\ref{torus1}, this inequality would be a consequence of 
\begin{multline}\label{4.73}
\Bigl| \, \sum_{j,k}\int_M \int_M
\frac{\tilde \chi_\la(\la_j)-\tilde \chi_\la(\tau_k)}{\la_j^2-\tau_k^2} e_j^0(x)e_j^0(y)V(y)e_{\tau_k}(x)
e_{\tau_k}(y) \, dx dy \Bigr|
\\
\lesssim  \begin{cases} \|V\|_{L^1(M)} \la^{\frac16-\frac b2} \log \la \,\,\,\, \text{if} \,\,\, n=2\\
\|V\|_{L^p(M)} \la^{k(b,p)} \log \la  \,\,\,\, \text{if} \,\,\, n\geq 3,
\end{cases}
\end{multline}
where $\, \tilde \chi_\la(\tau)$ is defined as in \eqref{4.5'} with $\e=\la^{-\frac{n-1}{n+1}}$.

And similarly, to obtain \eqref{1.15}, if we repeat the arguments in the proof of Theorem~\ref{torusthm1}, it 
suffices to show that
\begin{multline}\label{4.74}
\Bigl| \, \sum_{j,k}\int_M \int_M
\frac{\ala(\la_j)-\ala(\tau_k)}{\la_j^2-\tau_k^2} e_j^0(x)e_j^0(y)V(y)e_{\tau_k}(x)
e_{\tau_k}(y) \, dx dy\, \Bigr| 
\\
\lesssim 
 \begin{cases} \|V\|_{L^1(M)} \la^{\frac16} \log \la \,\,\,\, \text{if} \,\,\, n=2\\
\|V\|_{L^p(M)} \la^{k(p)} \log \la \,\,\,\, \text{if} \,\,\, n\geq 3,
\end{cases}
\end{multline}
where $\ala(\tau)$ is defined as in \eqref{2.15} with $\e=\la^{-\frac{n-1}{n+1}}$, and $k(p)=k(a,p)=\frac{n-1+a}{2}-1+\frac{n-1-a}{2}\cdot (2-\frac2p)+\frac n2\cdot(\frac2p-1)$,
if $a=\frac{n-1}{n+1}$. Note that by a straightforward calculation, when $p>\frac{2n}{n+2}$, the right side of \eqref{4.74} is controlled by the right side
of \eqref{1.15}.

As before, since the proofs of \eqref{4.73} and \eqref{4.74} are similar, we shall only give the details of \eqref{4.74} here. By using the same argument as in Section 3, the terms for large or small
frequencies $\tau_k$ in \eqref{4.74} will only contribute $C_V  \, 
\la^{n-2}(\log \la)^{1/2}$ to the right side. So the proof of \eqref{4.74} would be complete if we could establish the following.

\begin{proposition}\label{compprop2}
As in Theorem~\ref{torusthm}, fix $p>\frac{2n}{n+2}$ and assume that $V\in L^p(M)$, $V^{-}\in{\mathcal K}(M)$ for $n\geq 3$, and
$V\in \mathcal{K}(M)$ for $n=2$. If  $\ala(\tau)$ is defined as in \eqref{2.15} with $\e=\la^{-\frac{n-1}{n+1}}$, then
\begin{multline}\label{4.75}
\Bigl| \, \sum_{j}
\sum_{\{k: \, \tau_k\in [\lambda/2,10\lambda]\}}\int_M \int_M
\frac{\ala(\la_j)-\ala(\tau_k)}{\la_j^2-\tau_k^2} e_j^0(x)e_j^0(y)V(y)e_{\tau_k}(x)
e_{\tau_k}(y) \, dx dy\, \Bigr| 
\\
\le 
\begin{cases} C_V \la^{\frac16} \log \la \,\,\,\, \text{if} \,\,\, n=2\\
C_V \la^{k(p)} \log \la \,\,\,\, \text{if} \,\,\, n\geq 3.
\end{cases}
\end{multline}
\end{proposition}

To prove Proposition~\ref{compprop2}, we shall follow the same setup as in \eqref{3.2}-\eqref{3.5}, i.e., we need to show that
\begin{multline}\label{f.1}\tag{4.38$'$}
\Bigl| \sum_{\tau_k \in [\la/2,10\la]}\iint
K_{\tau_k}(x,y)e_{\tau_k}(x)e_{\tau_k}(y) \, V(y)\, dxdy\Bigr| \\
\lesssim  
\begin{cases} \|V\|_{L^1(M)} \la^{\frac16} \log \la \,\,\,\, \text{if} \,\,\, n=2\\
\|V\|_{L^p(M)} \la^{k(p)} \log \la \,\,\,\, \text{if} \,\,\, n\geq 3.
\end{cases}
\end{multline}
where we shall write $K_{\tau}(x,y)$ as in \eqref{3.3} for $\tau=\tau_k\in [\la/2,\la]$ and \eqref{3.4}
for $\tau=\tau_k\in (\la,10\la]$ separately.

As before, we just need to prove \eqref{f.1} when $K_{\tau}(x,y)$ is replaced by $K_{\tau, \ell_0}(x,y)$, $ \sum_{\ell} K^-_{\tau, \ell}(x,y)$, and $\sum_{\ell} R_{\tau, \ell}(x,y)(1-\ala(\tau))$ 
 if $\tau=\tau_k\in [\la/2,\la]$, and when $K_{\tau}(x,y)$ is replaced by $K_{\tau, \ell_0}(x,y)$, $ \sum_{\ell} K^+_{\tau, \ell}(x,y)$, and $\sum_{\ell} R_{\tau, \ell}(x,y)\ala(\tau)$ if $\tau=\tau_k\in (\la,10\la]$,
 since as in the proof of Proposition~\ref{compprop}, the other terms only contribute terms which are $O( \|V\|_{L^1(M)}\la^{n-2})$ to the right side of \eqref{f.1}.

First, by \eqref{4.41} in Lemma~\ref{lemma4.6}  and applying Lemma~\ref{delta} with
$\delta=2^\ell$ $m(\tau, x)=K^\pm_{\tau_k,\ell} (x,y)$ and $a_k=e_{\tau_k}(y)$, we have for $n=2$, $\la^{-a}<2^\ell\le \la/100$,
\begin{align}\label{4.76}
\Bigl| &\sum_{\tau_k\in I^\pm_{\ell,j}\cap[\la/2,10\la]}
\iint K^\pm_{\tau_k,\ell}(x,y)e_{\tau_k}(x)e_{\tau_k}(y)
V(y)\, dydx\Bigr| 
\\
&\le \|V\|_{L^1}
\cdot \sup_y \Bigl\|
\sum_{\tau_k\in I^\pm_{\ell,j}\cap[\la/2,10\la]}
 K^\pm_{\tau_k,\ell}(x,y)e_{\tau_k}(x)e_{\tau_k}(y)
\Bigr\|_{L^1(dx)}
\notag
\\
&\lesssim \|V\|_{L^1} \cdot \sup_y
\Bigl( \|K^\pm_{\la+j2^\ell,\ell}(\, \cdot \, ,y)\|_{L^2(M)}+\int_{I^\pm_{\ell,j}}\bigl\|\tfrac\partial{\partial \tau}K^\pm_{s,\ell}(\, \cdot \, y)\|_{L^2(M)} \, ds
\Bigr) \notag
\\
&\qquad \qquad \qquad \qquad \qquad \times \bigl(\sum_{\tau_k\in I^\pm_{\ell,j}\cap [\la/2,10\la]}
|e_{\tau_k}(y)|^2\bigr)^{1/2}
\notag
\\
&\lesssim \|V\|_{L^1} \la^{\frac{2-1}2-1}2^{-\ell/2}(1+j)^{-N}
\bigl(\sum_{\tau_k\in I^\pm_{\ell,j}\cap [\la/2,10\la]}
|e_{\tau_k}(y)|^2\bigr)^{1/2}
\notag
\\
&\lesssim \|V\|_{L^1} \la^{-1/2}2^{-\ell/2}(1+j)^{-N}
\bigl(\sum_{\mu \in {\mathbb N}\cap (I^\pm_{\ell,j}
\cap [\la/2,10\la]} \mu^{n-1}\bigr)^{1/2}
\notag
\\
&\lesssim \|V\|_{L^1} \max\{2^{-\ell/2},1\}\cdot(1+j)^{-N}.  \notag
\\
&\lesssim \|V\|_{L^1} \la^{\frac16}(1+j)^{-N}.  \notag
\end{align}
In the second to last inequality we used \eqref{4.71} and the fact that $|I^\pm_{\ell,j}|=2^\ell> \la^{-1/3}$ when $n=2$.

Similarly, for $n\ge 3$, by applying Lemma~\ref{delta}
\begin{align}\label{4.77}
\Bigl| &\sum_{\tau_k\in I^\pm_{\ell,j}\cap[\la/2,10\la]}
\iint K^\pm_{\tau_k,\ell}(x,y)e_{\tau_k}(x)e_{\tau_k}(y)
V(y)\, dydx\Bigr| 
\\
&\le \|V\|_{L^p}
\cdot  \Bigl\|
\sum_{\tau_k\in I^\pm_{\ell,j}\cap[\la/2,10\la]}
 K^\pm_{\tau_k,\ell}(x,y)e_{\tau_k}(x)e_{\tau_k}(y)
\Bigr\|_{L^{p^\prime}\big(dy; L^1(dx)\big)}
\notag
\\
&\lesssim \|V\|_{L^p} \cdot \sup_y
\Bigl( \|K^\pm_{\la+j2^\ell,\ell}(\, \cdot \, ,y)\|_{L^2(M)}+\int_{I^\pm_{\ell,j}}\bigl\|\tfrac\partial{\partial \tau}K^\pm_{s,\ell}(\, \cdot \, y)\|_{L^2(M)} \, ds
\Bigr) \notag
\\
&\qquad \qquad \qquad \qquad \qquad \times \bigl(\int_M\big(\sum_{\tau_k\in I^\pm_{\ell,j}\cap [\la/2,10\la]}
|e_{\tau_k}(y)|^2\big)^{p^\prime/2}dy\bigr)^{1/p^{\prime}}
\notag
\\
&\lesssim  \|V\|_{L^p}\la^{\frac{n-1}2-1}2^{-\ell/2}(1+j)^{-N}
\|\big(\sum_{\tau_k\in I^\pm_{\ell,j}\cap [\la/2,10\la]}
|e_{\tau_k}(y)|^2\big)^{1/2}\|_{L^{p^\prime}(M)},
\notag
\end{align}
where $\frac1p+\frac{1}{p^{\prime}}=1$.

Now in view of Corollary~\ref{err2} and \eqref{4.26}, since $\frac{1}{p^\prime}=\frac12\cdot\frac{2}{p^\prime}+\frac1\infty\cdot(1-\frac{2}{p^\prime})$, 
by H\"{o}lder's inequality, we have for all $\la>1$ and $2\le p'\le \infty$
$$\|\big(\sum_{\tau_k [\la,\la+\la^{-a}]}
|e_{\tau_k}(y)|^2\big)^{1/2}\|_{L^{p^\prime}(M)} \lesssim \la^{\frac{n-1-a}{2}\frac{2}{p^\prime}+\frac n2(1-\frac{2}{p^\prime})}.
$$
Since the number of intervals in $ I^\pm_{\ell,j}\cap [\la/2,10\la]$ with length comparable to $\la^{-a}$ is about $2^\ell \la^a$, by Minkowski's inequality
$$\|\big(\sum_{\tau_k\in I^\pm_{\ell,j}\cap [\la/2,10\la]}
|e_{\tau_k}(y)|^2\big)^{1/2}\|_{L^{p^\prime}(M)}\lesssim 2^{\ell/2}\la^{a/2}\la^{\frac{n-1-a}{2}\frac{2}{p^\prime}+\frac n2(1-\frac{2}{p^\prime})}.
$$
So the right side of \eqref{4.77} is bounded by $\la^{k(p)}(1+j)^{-N} \|V\|_{L^p}$.

If we sum over $j=0,1,2,\dots$, we see that \eqref{4.76} and \eqref{4.77} yields that for $\la^{-a}< 2^\ell \le \la/100$
\begin{multline}\label{4.78}
\Big| \sum_{\la <\tau_k\le 10\la}
\iint K^+_{\tau_k,\ell}(x,y) e_{\tau_k}(x)e_{\tau_k}(y) 
\, V(y)\, dxdy\Bigr|
\\
+\Big| \sum_{\la/2 \le \tau_k\le\la}
\iint K^-_{\tau_k,\ell}(x,y) e_{\tau_k}(x)e_{\tau_k}(y) 
\, V(y)\, dxdy\Bigr|
\lesssim 
\begin{cases} \|V\|_{L^1(M)} \la^{\frac16}  \,\,\,\, \text{if} \,\,\, n=2\\
\|V\|_{L^p(M)} \la^{k(p)}  \,\,\,\, \text{if} \,\,\, n\geq 3.
\end{cases}.
\end{multline}
If we take $\delta=\la^{-a}$ in Lemma~\ref{delta}, this argument
also gives
\begin{multline}\label{4.79}
\Bigl| \sum_{\la/2 \le \tau_k\le \la}
\iint K_{\tau_k,0}(x,y) e_{\tau_k}(x)e_{\tau_k}(y) \, 
V(y)\, dx dy\Bigr|
\\
+ \Bigl| \sum_{\la < \tau_k\le 10\la}
\iint K_{\tau_k,0}(x,y) e_{\tau_k}(x)e_{\tau_k}(y) \, 
V(y)\, dx dy\Bigr|
\lesssim 
\begin{cases} \|V\|_{L^1(M)} \la^{\frac16}  \,\,\,\, \text{if} \,\,\, n=2\\
\|V\|_{L^p(M)} \la^{k(p)}  \,\,\,\, \text{if} \,\,\, n\geq 3.
\end{cases}.
\end{multline}

Next, since $R_{\tau,\ell}$ enjoys the bounds in \eqref{4.45}, we can repeat the arguments in 
\eqref{4.76} and \eqref{4.77} to see that for $\la^{-a}< 2^\ell \le \la/100$ we
have for $n=2$,
\begin{align*}
\Bigl| \sum_{\tau_k\in I^+_{\ell,j}\cap (\la,10\la]}
&\iint R_{\tau_k,\ell}(x,y)\ala(\tau_k)e_{\tau_k}(x)
e_{\tau_k}(y) V(y)\, dxdy\Bigr|
\\
&\lesssim \|V\|_{L^1}\cdot 2^{-\ell/2}\la^{\frac{2-1}2-1}
\sup_y\bigl(\sum_{\tau_k\in I^+_{\ell,j}\cap
(\la,10\la]} |\ala(\tau_k)e_{\tau_k}(y)|^2\bigr)^{1/2} \\
&\lesssim 
\la^{\frac16}(1+j)^{-N}\|V\|_{L^1},
\end{align*}
since $\ala(\tau_k)=O((1+j)^{-N})$ if $\tau_k
\in I^+_{\ell,j}$. 

Similarly for $n\geq 3$, we have
\begin{align*}
\Bigl| \sum_{\tau_k\in I^+_{\ell,j}\cap (\la,10\la]}
&\iint R_{\tau_k,\ell}(x,y)\ala(\tau_k)e_{\tau_k}(x)
e_{\tau_k}(y) V(y)\, dxdy\Bigr|
\\
&\lesssim \|V\|_{L^p}\cdot 2^{-\ell/2}\la^{\frac{n-1}2-1}
\|\big(\sum_{\tau_k\in I^\pm_{\ell,j}\cap [\la/2,10\la]}
|\ala(\tau_k)e_{\tau_k}(y)|^2\big)^{1/2}\|_{L^{p^\prime}(M)} \\
&\lesssim 
\la^{k(p)}(1+j)^{-N}\|V\|_{L^p} .
\end{align*}

 Summing over this bound over $j$
of course yields
\begin{multline}\label{4.80}
\Bigl|\sum_{\la<\tau_k\le 10\la}
\iint R_{\tau_k,\ell}(x,y)
\ala(\tau_k)e_{\tau_k}(x)e_{\tau_k}(y) V(y)\, dxdy\Bigr| \\
\lesssim 
\begin{cases} \|V\|_{L^1(M)} \la^{\frac16}  \,\,\,\, \text{if} \,\,\, n=2\\
\|V\|_{L^p(M)} \la^{k(p)}  \,\,\,\, \text{if} \,\,\, n\geq 3.
\end{cases}.
\end{multline}
The same argument gives
\begin{multline}\label{4.81}
\Bigl|\sum_{\la/2\le\tau_k\le \la}
\iint R_{\tau_k,\ell}(x,y)
\bigl(1-\ala(\tau_k)\bigr)e_{\tau_k}(x)e_{\tau_k}(y) V(y)\, dxdy\Bigr| \\
\lesssim  
\begin{cases} \|V\|_{L^1(M)} \la^{\frac16}  \,\,\,\, \text{if} \,\,\, n=2\\
\|V\|_{L^p(M)} \la^{k(p)}  \,\,\,\, \text{if} \,\,\, n\geq 3.
\end{cases}.
\end{multline}

We now have assembled all the ingredients for the proof
of \eqref{f.1}.  If we use \eqref{4.78}, \eqref{4.79},
\eqref{4.81}, \eqref{4.81}, \eqref{3.21} and \eqref{3.22} along with
\eqref{3.3}, we conclude that the analog of \eqref{f.1}
must be valid where the sum is taken over
$\tau_k\in [\la/2,\la]$.  The log-loss comes from the fact
that there are $\approx \log \la$ terms
$K^-_{\tau,\ell}$ and $R_{\tau,\ell}$.
We similarly obtain the analog of \eqref{f.1} where the sum is taken over $\tau_k\in (\la,10\la]$ from \eqref{3.4} along with
\eqref{4.78}, \eqref{4.79}, \eqref{3.18}, \eqref{4.80} and \eqref{3.21}. So the proof of \eqref{4.75} is complete.

\end{proof}

\bibliography{refs3}
\bibliographystyle{abbrv}
\end{document}